\documentclass[10pt]{amsart}
\usepackage{amssymb, amsmath, amsthm, amsfonts,mathtools}
\usepackage{stmaryrd}
\usepackage{bbm}
\usepackage{mathrsfs}
\usepackage{graphicx}
\usepackage{eqparbox}
\usepackage{tikz}
\usepackage{tikz-cd}
\usepackage{todonotes}
\usepackage{float}
\usepackage{hvfloat}
\usepackage{caption}
\usepackage{pdflscape}
\usepackage{notoccite}
\usepackage{hyperref}  
\usepackage{cleveref}
\usepackage{url}
\usepackage[all,arc,2cell]{xy}
\UseAllTwocells
\usepackage{enumerate}
\usepackage{bbm}
\usepackage{enumitem}

\DeclareMathSymbol{\square}{\mathbin}{AMSa}{"03}

\usepackage{geometry}

\hypersetup{%
  bookmarks=true,%
  colorlinks=true,%
  linkcolor=blue,%
  citecolor=blue,%
  filecolor=blue,%
  menucolor=blue,%
  urlcolor=blue,%
  pdfnewwindow=true,%
  pdfstartview=FitBH}

\def\@url#1{{\tt\def~{\lower3.5pt\hbox{\char'176}}\def\_{\char'137}#1}}

\def\makeautorefname#1#2{\expandafter\def\csname#1autorefname\endcsname{#2}}
\makeautorefname{equation}{Equation}%
\makeautorefname{footnote}{footnote}%
\makeautorefname{item}{item}%
\makeautorefname{figure}{Figure}%
\makeautorefname{table}{Table}%
\makeautorefname{part}{Part}%
\makeautorefname{appendix}{Appendix}%
\makeautorefname{chapter}{Chapter}%
\makeautorefname{section}{Section}%
\makeautorefname{subsection}{Section}%
\makeautorefname{subsubsection}{Section}%
\makeautorefname{paragraph}{Paragraph}%
\makeautorefname{subparagraph}{Paragraph}%
\makeautorefname{theorem}{Theorem}%
\makeautorefname{theo}{Theorem}%
\makeautorefname{thm}{Theorem}%
\makeautorefname{goal}{Goal}%
\makeautorefname{addendum}{Addendum}%
\makeautorefname{addend}{Addendum}%
\makeautorefname{add}{Addendum}%
\makeautorefname{maintheorem}{Main theorem}%
\makeautorefname{mainthm}{Main theorem}%
\makeautorefname{corollary}{Corollary}%
\makeautorefname{claim}{Claim}%
\makeautorefname{corol}{Corollary}%
\makeautorefname{coro}{Corollary}%
\makeautorefname{cor}{Corollary}%
\makeautorefname{lemma}{Lemma}%
\makeautorefname{lemm}{Lemma}%
\makeautorefname{lem}{Lemma}%
\makeautorefname{sublemma}{Sublemma}%
\makeautorefname{sublem}{Sublemma}%
\makeautorefname{subl}{Sublemma}%
\makeautorefname{proposition}{Proposition}%
\makeautorefname{proposit}{Proposition}%
\makeautorefname{propos}{Proposition}%
\makeautorefname{propo}{Proposition}%
\makeautorefname{prop}{Proposition}%
\makeautorefname{property}{Property}
\makeautorefname{proper}{Property}
\makeautorefname{scholium}{Scholium}%
\makeautorefname{step}{Step}%
\makeautorefname{conjecture}{Conjecture}%
\makeautorefname{conject}{Conjecture}%
\makeautorefname{conj}{Conjecture}%
\makeautorefname{question}{Question}
\makeautorefname{questn}{Question}
\makeautorefname{quest}{Question}
\makeautorefname{ques}{Question}
\makeautorefname{qn}{Question}
\makeautorefname{definition}{Definition}%
\makeautorefname{defin}{Definition}%
\makeautorefname{defi}{Definition}%
\makeautorefname{def}{Definition}%
\makeautorefname{defn}{Definition}%
\makeautorefname{dfn}{Definition}%
\makeautorefname{notation}{Notation}
\makeautorefname{nota}{Notation}
\makeautorefname{notn}{Notation}
\makeautorefname{remark}{Remark}%
\makeautorefname{rema}{Remark}%
\makeautorefname{rem}{Remark}%
\makeautorefname{rmk}{Remark}%
\makeautorefname{rk}{Remark}%
\makeautorefname{remarks}{Remarks}%
\makeautorefname{rems}{Remarks}%
\makeautorefname{rmks}{Remarks}%
\makeautorefname{rks}{Remarks}%
\makeautorefname{example}{Example}%
\makeautorefname{examp}{Example}%
\makeautorefname{exmp}{Example}%
\makeautorefname{exmps}{Examples}%
\makeautorefname{exam}{Example}%
\makeautorefname{exa}{Example}%
\makeautorefname{algorithm}{Algorith}%
\makeautorefname{algo}{Algorith}%
\makeautorefname{alg}{Algorith}%
\makeautorefname{axiom}{Axiom}%
\makeautorefname{axi}{Axiom}%
\makeautorefname{ax}{Axiom}%
\makeautorefname{case}{Case}%
\makeautorefname{claim}{Claim}%
\makeautorefname{clm}{Claim}%
\makeautorefname{assumption}{Assumption}%
\makeautorefname{assumpt}{Assumption}%
\makeautorefname{conclusion}{Conclusion}%
\makeautorefname{concl}{Conclusion}%
\makeautorefname{conc}{Conclusion}%
\makeautorefname{condition}{Condition}%
\makeautorefname{condit}{Condition}%
\makeautorefname{cond}{Condition}%
\makeautorefname{construction}{Construction}%
\makeautorefname{construct}{Construction}%
\makeautorefname{const}{Construction}%
\makeautorefname{cons}{Construction}%
\makeautorefname{criterion}{Criterion}%
\makeautorefname{criter}{Criterion}%
\makeautorefname{crit}{Criterion}%
\makeautorefname{exercise}{Exercise}%
\makeautorefname{exer}{Exercise}%
\makeautorefname{exe}{Exercise}%
\makeautorefname{problem}{Problem}%
\makeautorefname{problm}{Problem}%
\makeautorefname{probm}{Problem}%
\makeautorefname{prob}{Problem}%
\makeautorefname{rec}{recollection}%
\makeautorefname{solution}{Solution}%
\makeautorefname{soln}{Solution}%
\makeautorefname{sol}{Solution}%
\makeautorefname{summary}{Summary}%
\makeautorefname{summ}{Summary}%
\makeautorefname{sum}{Summary}%
\makeautorefname{operation}{Operation}%
\makeautorefname{oper}{Operation}%
\makeautorefname{observation}{Observation}%
\makeautorefname{observn}{Observation}%
\makeautorefname{obser}{Observation}%
\makeautorefname{obs}{Observation}%
\makeautorefname{ob}{Observation}%
\makeautorefname{convention}{Convention}%
\makeautorefname{convent}{Convention}%
\makeautorefname{conv}{Convention}%
\makeautorefname{cvn}{Convention}%
\makeautorefname{warning}{Warning}%
\makeautorefname{warn}{Warning}%
\makeautorefname{note}{Note}%
\makeautorefname{fact}{Fact}%

  \makeatletter
                   \let\c@lemma\c@theorem
                  \makeatother

 \numberwithin{equation}{section}

\newtheorem{thm}[equation]{Theorem}

\newtheorem{prop}[equation]{Proposition}
\newtheorem{lem}[equation]{Lemma}

\newtheorem*{cor*}{Corollary}

\theoremstyle{definition}
\newtheorem{defn}[equation]{Definition}
\newtheorem{exmp}[equation]{Example}

\newtheorem{rmk}[equation]{Remark}

\newtheorem{notation}[equation]{Notation}

\makeatletter
\let\c@lem=\c@thm
\let\c@cor=\c@thm
\let\c@prop=\c@thm
\let\c@lem=\c@thm
\let\c@defn=\c@thm
\let\c@exmp=\c@thm
\let\c@exmps=\c@thm
\let\c@rem=\c@thm
\let\c@rec=\c@thm
\let\c@warn=\c@thm
\let\c@claim=\c@thm
\let\c@quest=\c@thm
\makeatother

\newcommand{\Z}{\mathbb{Z}}

\newcommand{\W}{\mathbb{W}}
\newcommand{\sphere}{\mathbb{S}}

\DeclareSymbolFontAlphabet{\scr}{rsfs}

\newcommand{\sma}{\mathbin{\wedge}}

\def\quickop#1{\expandafter\newcommand\csname #1\endcsname{\operatorname{#1}}}
\quickop{Hom} \quickop{End} \quickop{Aut} \quickop{Tel} \quickop{Mic} 
\quickop{Ext} \quickop{Tor} \quickop{Cotor} \quickop{Id} \quickop{Coker} \quickop{Ker}
\quickop{Lim} \quickop{Colim} \quickop{Holim} \quickop{Hocolim}
\quickop{id} \quickop{tel} \quickop{mic} \quickop{coker} 
\quickop{colim}

\DeclareMathOperator*{\holim}{holim}

\DeclareMathOperator{\sd}{sd}

\newcommand{\cy}{\text{cy}}

\DeclareMathOperator{\THH}{THH}

\DeclareMathOperator{\TR}{TR}
\DeclareMathOperator{\HH}{HH}
\newcommand{\cyc}{\mathrm{cyc}}
\newcommand{\Mack}{\mathrm{Mack}}
\newcommand{\Green}{\mathrm{Green}}
\newcommand{\Tamb}{\mathrm{Tamb}}

\newcommand{\m}[1]{\underline{#1}}
\newcommand{\Spectra}{\mathrm{Sp}}

\newcommand{\Set}{\mathrm{Set}}

\newcommand{\Ab}{\mathrm{Ab}}

\newcommand{\res}{\mathrm{res}}

\newcommand{\tr}{\mathrm{tr}}

\definecolor{darkspringgreen}{rgb}{0.09, 0.45, 0.27}
\definecolor{darkterracotta}{rgb}{0.8, 0.31, 0.36}
	\definecolor{darkcoral}{rgb}{0.8, 0.36, 0.27}
	\definecolor{indiagreen}{rgb}{0.07, 0.53, 0.03}
	\definecolor{mountainmeadow}{rgb}{0.19, 0.73, 0.56}
	\definecolor{mountbattenpink}{rgb}{0.6, 0.48, 0.55}
	\definecolor{palatinatepurple}{rgb}{0.41, 0.16, 0.38}
	\definecolor{cinnamon}{rgb}{0.82, 0.41, 0.12}
	\definecolor{chocolate}{rgb}{0.82, 0.41, 0.12}

\usepackage{xcolor}
\definecolor{seagreen}{RGB}{46,139,87}
\definecolor{maroon}{RGB}{128,0,0}
\definecolor{darkviolet}{RGB}{210,150,180}

\usepackage{scalerel}

\renewcommand{\id}{\mathrm{id}}
\newcommand{\op}{\mathrm{op}}
\newcommand{\TC}{\textnormal{TC}}

\DeclareMathOperator{\Span}{Span}
\newcommand{\Fin}{\mathrm{Fin}}

\DeclareMathOperator{\CoInd}{CoInd}

\newcommand{\xto}{\xrightarrow}

\usepackage[utf8]{inputenc}

\usepackage{adjustbox}

\DeclareFontFamily{U}{mathx}{\hyphenchar\font45}
\DeclareFontShape{U}{mathx}{m}{n}{
      <5> <6> <7> <8> <9> <10>
      <10.95> <12> <14.4> <17.28> <20.74> <24.88>
      mathx10
      }{}
\DeclareSymbolFont{mathx}{U}{mathx}{m}{n}

\DeclareMathSymbol{\bigboxvoid}{1}{mathx}{"DC}

\title{Equivariant Witt Complexes and Twisted Topological Hochschild Homology}

\author[Bohmann]{Anna Marie Bohmann}
\address[Bohmann]{Department of Mathematics, Vanderbilt University, Nashville, TN, 37240}
\email{am.bohmann@vanderbilt.edu}
\author[Gerhardt]{Teena Gerhardt}
\address[Gerhardt]{Department of Mathematics, Michigan State University, East Lansing, MI 48824 }
\email{teena@math.msu.edu}
\author[Krulewski]{Cameron Krulewski}
\address[Krulewski]{Department of Mathematics, Massachusetts Institute of Technology, Cambridge, MA 02139}
\email{camkru@mit.edu}
\author[Petersen]{Sarah Petersen}
\address[Petersen]{Department of Mathematics, University of Colorado Boulder, Boulder, CO 80309}
\email{sarahllpetersen@gmail.com}
\author[Yang]{Lucy Yang}
\address[Yang]{Department of Mathematics, Columbia University, New York, NY 10027}
\email{ly2620@columbia.edu}

\makeatletter
\@namedef{subjclassname@2020}{\textup{2020} Mathematics Subject Classification}
\makeatother

\begin{document}

\begin{abstract} 
The topological Hochschild homology of a ring (or ring spectrum) $R$ is an $S^1$-spectrum, and the fixed points of $\THH(R$) for subgroups $C_n \subset S^1$ have been widely studied due to their use in algebraic $K$-theory computations. Hesselholt and Madsen proved that the fixed points of topological Hochschild homology are closely related to Witt vectors \cite{HeMa97}. Further, they defined the notion of a Witt complex, and showed that it captures the algebraic structure of the homotopy groups of the fixed points of THH \cite{HeMa04}. Recent work \cite{AnBlGeHiLaMa} defines a theory of twisted topological Hochschild homology for equivariant rings (or ring spectra) that builds upon Hill, Hopkins and Ravenel's work on equivariant norms \cite{HHR}. In this paper, we study the algebraic structure of the equivariant homotopy groups of twisted THH. In particular, drawing on the definition of equivariant Witt vectors in \cite{BlGeHiLa}, we define an \emph{equivariant Witt complex} and prove that the equivariant homotopy of twisted THH has this structure.  Our definition of equivariant Witt complexes contributes to a growing body of research in the subject of equivariant algebra.
\end{abstract}
  
\maketitle

\section{Introduction}
Recent computations in higher algebraic $K$-theory owe much of their success to trace methods, an approach in which algebraic $K$-theory is approximated by more computable invariants receiving natural transformations, or \emph{trace maps}, from $K$-theory. The first such approximations to algebraic $K$-theory are Hochschild homology ($\HH$) and its topological analogue, topological Hochschild homology ($ \THH$). 

While trace theories such as $ \HH $ and $ \THH $ are often viewed as stepping stones to algebraic $K$-theory computations, they are interesting in their own right. 
For instance, the Hochschild--Kostant--Rosenberg theorem furnishes a filtration on the Hochschild homology of a ring, with associated graded given by the derived de Rham complex. Recently, Bhatt--Morrow--Scholze used filtrations on trace theories to define prismatic cohomology \cite{BMS}. On the geometric topology side of things, the topological Hochschild homology of the pointed loop space of a space $ X $ computes the stable homotopy type of the free loop space of $ X $.  Such applications provide many  motivations for studying the rich structure of trace theories in algebraic $K$-theory and beyond.

One of the most important structures enjoyed by topological Hochschild homology is a circle action.  This is essential
 to defining topological cyclic homology (TC), a close approximation to algebraic $K$-theory. In the classical approach to trace methods \cite{BHM}, TC is defined using fixed points of THH under the actions of finite subgroups of $S^1$.  
 Given the key role that these fixed points play in computing algebraic $K$-theory, understanding the algebraic structure arising in their homotopy groups is of great interest. 

Hesselholt and Madsen proved in \cite[Theorem F]{HeMa97} that for a commutative ring $A$ and a prime $p$, 
\[
\pi_0(\THH(A)^{C_{p^n}}) \cong W_{n+1}(A),
\]
where $ W_{n+1}(A)$ denotes the length $n+1$ $p$-typical Witt vectors of $A$ and $C_{p^n}\subset S^1$ is the cyclic group of order $p^n$. Further, they prove in \cite[Section 2]{HeMa04} that for $p$ an odd prime and $A$ a $\mathbb{Z}_{(p)}$-algebra, the pro-differential graded ring $E^*_{\bullet} =  \pi_*(\THH(A)^{C_{p^{\bullet -1}}})$ has the algebraic structure of what they term a \emph{Witt complex}. This rigid algebraic structure on the equivariant homotopy groups of THH has facilitated numerous calculations of these homotopy groups, and consequently of algebraic $K$-theory. See, for example, \cite{BokstedtMadsen, HeMa97, HeMa03, BokstedtMadsen2, AGH_Ktrunc, AnGe11}.

In recent years, several equivariant analogues of classical topological Hochschild homology have emerged, including Real topological Hochschild homology (THR) and twisted topological Hochschild homology for $C_n$-spectra (THH$_{C_n}$), the latter of which is the focus of this work. Twisted topological Hochschild homology arose as an extension of Hill, Hopkins and Ravenel's work on equivariant norms. In their work on the Kervaire Invariant One problem, they developed multiplicative norm functors from $H$-spectra to $G$-spectra, for $H\subset G$ finite groups \cite{HHR}. Building upon this work, Angeltveit, Blumberg, Gerhardt, Hill, Lawson and Mandell \cite{AnBlGeHiLaMa} extended the norm construction from finite groups to $S^1$ (see also \cite{BDS18}), and showed that topological Hochschild homology arises an an equivariant norm. Indeed, for a ring spectrum $R$, $\THH(R)$ is $N_e^{S^1}(R)$.  The authors then generalize this work to define the $C_n$-twisted topological Hochschild homology of $R$, $\THH_{C_n}(R)$, as the equivariant norm $N_{C_n}^{S^1}(R)$.   They also give an explicit construction of this norm via a twisted cyclic bar construction.  
As discussed in \cite{AGHKK2}, twisted THH is related via trace maps to Merling's equivariant algebraic $K$-theory \cite{Merling}. Such connections are explored further in forthcoming work of Chan, Gerhardt, and Klang \cite{CGK}.

While topological Hochschild homology was conceived as a topological enhancement of the classical algebraic theory of Hochschild homology, twisted THH preceded its algebraic counterpart.  Nevertheless, in \cite{BlGeHiLa} the authors define a ``twisted'' version of Hochschild homology.  This is no longer simply an invariant of rings, but rather of Green functors, the equivariant algebraic analogue of rings.  Blumberg, Gerhardt, Hill and Lawson use this theory to define Witt vectors for Green functors, and show that these Witt vectors capture the equivariant $\m{\pi}_0$ of twisted THH. In this equivariant analogue of Hesselholt and Madsen's theorem, they show that for $H\subset G \subset S^1$ finite subgroups and $\m{R}$ an $H$-Tambara functor, 
\[
\m{\pi}_0^G\THH_H(\m{R}) \cong \m{W}_G(\m{R}).
\]
Thus, equivariant Witt vectors capture the zeroth homotopy groups of twisted THH.  However, as in the classical case, one would like to understand the algebraic structure that arises in the homotopy groups of twisted THH in all dimensions.  In the current work, we show that the higher equivariant homotopy groups of twisted THH assemble into a structure which we axiomatize as an \emph{equivariant Witt complex} (Definition \ref{defn:equivariant_Witt_complex}). More concretely, we prove the following. 

\begin{thm}\label{thm:main}
    Let $ n $ be a positive integer and $ p $ an odd prime not dividing $ n $. 
  	For $\m{R}$ a $C_n$-Tambara functor such that $\m{R}(C_n/C_m)$ is a $\mathbb{Z}_{(p)}$-algebra for each $C_m\subset C_n$, the graded Green functors $ \{\underline{\pi}_*^{C_{p^mn}} \THH_{C_n} \m{R}\}_{m \geq 0} $ form an equivariant Witt complex over $\m{R}$.
\end{thm}

We also show that when $n=1$, our $C_n$-equivariant Witt complex recovers Hesselholt and Madsen's classical notion of a Witt complex. 
In the classical (non-equivariant) setting, the de Rham--Witt complex for a $\mathbb{Z}_{(p)}$-algebra $A$ arises as the initial object in the category of Witt complexes over $A$. We will return to the development of an equivariant de Rham-Witt complex for $\m{R}$ a $C_n$-Tambara functor in subsequent work.

Our development of an equivariant Witt complex adds to a growing body of work on \emph{equivariant algebra}, i.e.~the algebra that arises in the study of equivariant homotopy theory. The current work also includes further development of the theory of equivariant Witt vectors. In particular, for a $C_n$-Tambara functor $\m{R}$ we construct multiplicative lifts 
\[
[-]_k\colon  \underline{R}(C_n / C_m) \to \underline{\W}_{C_{p^k n}} (\underline{R}) (C_{p^k n} / C_{p^km})
\]
and prove that these maps lift the $p^k$-power maps when $p$ is coprime to $n$.

\subsection{Organization} 

In Section \ref{sec:background} we recall the classical relationship between topological Hochschild homology and Witt vectors. In Section \ref{equivariantHH} we move to the equivariant setting, giving an overview of the basic objects in equivariant algebra and recalling the constructions of twisted topological Hochschild homology of equivariant spectra and Hochschild homology of Green functors. Section \ref{sec:EquivariantWitt} focuses on equivariant Witt vectors. We recall the definition of equivariant Witt vectors from \cite{BlGeHiLa}, as well as expand upon the theory by defining multiplicative lift maps in the equivariant setting. This section also includes several concrete examples. In Section \ref{sec:EquivariantWittCplx} we define a $C_n$-equivariant Witt complex and prove that when $n=1$ this recovers Hesselholt and Madsen's classical notion of a Witt complex. The proof of Theorem \ref{thm:main} is in Section \ref{sec:TwistedTHHEquivariantWitt}, where we show that the equivariant homotopy groups of twisted THH have the structure of an equivariant Witt complex. 

\subsection{Acknowledgements} This paper is part of the authors' Women in Topology IV project. We are grateful to the organizers of the Women in Topology IV workshop, as well as to the Hausdorff Research Institute for Mathematics, where much of this research was carried out. The authors are grateful to David Chan, Mike Hill, John Rognes, Carissa Slone, and Ben Spitz for helpful conversations related to this work.  We also thank the anonymous referee for their helpful comments.

The first author was supported by NSF grant DMS-2104300. 
 The second author was supported by NSF grants DMS-2104233 and DMS-2404932. The third author was supported by the NSF under Grant Nos.~DGE-2141064 and DMS-2220741. The fourth author thanks the Max Planck Institute for Mathematics in Bonn for its hospitality and financial support during the initial stages of this project. 
The fourth author's contributions are also based on work supported by the National Science Foundation under Grant Nos.~DMS-2135884 and ~DMS-2220741. 
The fifth author was supported by an NSF Graduate Research Fellowship under Grant No.~DGE-2140743 during the completion of this work. 
The authors would like to thank Columbia University, Michigan State University, University of Colorado Boulder, and Vanderbilt University for their hospitality during various research visits. The first, second, and fourth authors would like to thank the Isaac Newton Institute for Mathematical Sciences, Cambridge, for support and hospitality during the programme ``Equivariant homotopy theory in context" where work on this paper was undertaken. The authors would also like to thank the Foundation Compositio Mathematica, the Foundation Nagoya Mathematical Journal, and the K-theory Foundation for financial support for the Women in Topology IV workshop.

\section{Topological Hochschild homology and Witt vectors}\label{sec:background}

Let $ A $ be a commutative ring. 
Classical results relating the values of topological trace theories on $ A $ with more algebro-geometric constructions have been instrumental to computations of algebraic $K$-theory. 
The central goal of this paper is to formulate and prove generalizations of these classical results to genuine equivariant analogues of commutative rings called Tambara functors. 
In this section, we provide the non-equivariant context for our work. We first introduce the relevant trace theories, topological Hochschild homology and topological restriction homology. 
Then we recall Witt vectors and the classical definition of Witt complexes for commutative rings. 

\subsection{Topological Hochschild homology} \label{introTHH}
We begin by recalling the classical definition of the topological Hochschild homology of a ring spectrum.  Let $\Spectra$ denote the category of spectra.

\begin{defn}\label{defn:THH_classical}
    Let $A$ be a ring spectrum. The \emph{topological Hochschild homology} $ \THH(A) $ of $A$ is the geometric realization of the simplicial spectrum known as the \emph{cyclic bar construction}
    \begin{align*}
    N^\cy_\bullet A \colon \Delta^\op &\to \Spectra \\
               [n] &\mapsto A^{\wedge n+1}.
    \end{align*}
 The face maps arise from the multiplication on $A$ as follows:
\[d_i = \begin{cases} \id^{\wedge i} \wedge \mu \wedge \id^{\wedge (n-i-1)} & \text{for } 0\leq i<n \\  (\mu \wedge \id^{\wedge (n-1)}) \circ \tau& \text{for }i=n.
\end{cases}
\] 
Here $\mu$ denotes the multiplication map of $A$. Note the use of the cyclic permuation $\tau$ that moves  the last factor to the front in the definition of the last face map.

The $ i $th degeneracy map $ s_i \colon A^{\wedge n} \to A^{\wedge n+1} $ is given by:
\[
s_i =  \id^{\wedge (i+1)} \wedge \eta \wedge \id^{\wedge (n-i)} \text{\quad for }  0\leq i \leq n
\] 
where $\eta$ is the unit map of $A$. 
\end{defn}

\begin{rmk}
For a classical ring $A$, there is an associated ring spectrum, $HA$, the Eilenberg--MacLane spectrum of $A$. The topological Hochschild homology of a classical ring $A$ is thus defined to be
\[
\THH(A) \coloneqq \THH(HA). 
\]
\end{rmk}

\begin{rmk} \label{rmk:THH_HH_relationship_classical}
  Topological Hochschild homology is the topological analogue of the classical theory of Hochschild homology. For a ring $A$, there is a linearization map 
  \[
  \pi_q(\THH(A)) \to \HH_q(A), 
  \]
  which factors the Dennis trace map from algebraic $K$-theory to Hochschild homology
  \[
  K_q(A) \to \pi_q(\THH(A)) \to \HH_q(A).
  \]
\end{rmk}

 While we have only specified the simplicial structure, the cyclic bar construction produces a cyclic spectrum, and thus $\THH(A)$ inherits a canonical action of $ S^1 $.  In fact, this action makes $\THH(A)$ into a genuine $S^1$-spectrum---this follows from B\"okstedt's original work on the subject \cite{bokstedt_thh}. In the case where $A$ is commutative, this action is especially straightforward.    The circle $ S^1 $ has a simplicial model with $ (n+1) $ simplices in dimension $ n $ and we may view the cyclic bar construction as $A\otimes S^1$ \cite{MSV}.    In fact, if $A$ is an $\mathbb{E}_\infty$-ring spectrum, this perspective allows one to alternatively define $ \THH(A) $  as the free $\mathbb{E}_\infty$-ring spectrum with an $ S^1 $-action (through $\mathbb{E}_\infty$-ring maps) receiving a map from $ A $. 

    In this paper, our perspective is shaped by a new approach to topological Hochschild homology: the topological Hochschild homology of a ring spectrum $A$ can be viewed as the norm $ N_e^{S^1}A $.  This perspective builds on work of Hill, Hopkins, and Ravenel on norms in equivariant homotopy theory \cite{HHR} and is developed in \cite{AnBlGeHiLaMa} and \cite{BDS18}.

Topological Hochschild homology has the further structure of a \emph{cyclotomic spectrum}.  This means there are  equivalences between the geometric fixed points of topological Hochschild homology, $\Phi^{C_n}\THH(A)$, and $\THH(A)$ itself
\[ \rho^*_n\Phi^{C_n}\THH(A) \xto{\simeq} \THH(A)\]
that are compatible as we vary the subgroup $C_n$.  Here $ \rho_n\colon S^1 \xto{\cong} S^1 / C_n$ is the $ n^{\textrm{th}}$ root isomorphism.

  More relevant to this work is the weaker notion of a $p$-\emph{cyclotomic} spectrum, which we now recall.   More details on  cyclotomic and $p$-cyclotomic spectra  can be found in \cite{BlumbergMandell-cycl}.
    \begin{defn} 
\label{defn:pcyclotomic}
   For a prime $p$, a \emph{$ p $-cyclotomic spectrum} is the data of a genuine $ S^1 $-spectrum $X$ together with a morphism of $S^1$-spectra
\[
    r_p\colon    \rho^*_p \Phi^{C_p} X \to X 
\]
    with the property that for all $n\geq 0$, the induced map of $C_{p^n}$-fixed point spectra is a weak equivalence. 
\end{defn}
 Having an action of $ S^1 $ also allows us to take fixed points with respect to finite subgroups of $ S^1 $: 
\begin{defn}\label{defn:TR_classical}
    Let $ A $ be a ring (or ring spectrum), and fix a prime $ p $. 
    The \emph{topological restriction homology} $ \{\TR^{n+1}(A;p)\}_{n \in \Z_{\geq 0}} $ of $ A $ is defined to be the categorical fixed points 
\[\TR^{n+1}(A;p) \coloneqq \THH(A)^{C_{p^n}} \] 
with respect to the finite subgroups $ C_{p^n}\subset S^1 $ as $ n $ varies.  We denote $\pi_q(\THH(A)^{C_{p^{n}}})$ by $\TR^{n+1}_q(A;p)$.
\end{defn}

The fixed points of topological Hochschild homology are related via several operators. 
Inclusion of fixed points induces a map, called the Frobenius:
\[
F\colon \THH(A)^{C_{p^n}} \to \THH(A)^{C_{p^{n-1}}}.
\]
The $p$-cyclotomic structure on $\THH(A)$ induces a second map, called the restriction:
\[
R\colon \THH(A)^{C_{p^n}} =  (\rho_p^*\THH(A)^{C_{p}})^{C_{p^{n-1}}} \to (\rho_p^*\Phi^{C_p} \THH(A))^{C_{p^{n-1 }}} \xrightarrow{\simeq} \THH(A)^{C_{p^{n-1}}}.
\]
The first map in this composite is induced by the canonical map from categorical fixed points to geometric fixed points, and the latter map is the map $r_p$ from the $p$-cyclotomic structure. Work of B\"okstedt, Hsiang, and Madsen \cite{BHM} then defines topological cyclic homology.

\begin{defn}
    Let $ A $ be a ring (or ring spectrum), and fix a prime $ p $.  The \emph{topological cyclic homology} of $ A $ is the homotopy limit
    \[
    \TC(A;p) = \holim_{R,F} \THH(A)^{C_{p^n}}\,. 
    \]
\end{defn}
By construction, the fixed points 
\[
\pi_q(\THH(A)^{C_{p^{n}}})
\]
interpolate between $ \THH $ and $ \TC $. 
Since the Dennis trace map $ K \to \THH $ factors through topological cyclic homology \cite[\S5]{BHM}, $ \TC $ and $ \TR $ are closer approximations to algebraic $K$-theory than $ \THH $. 
Topological restriction homology $\pi_q(\THH(A)^{C_{p^n}})$ has been computed for many examples and used to understand algebraic $K$-theory \cite{BokstedtMadsen, HeMa97, HeMa03, HeMa04, AG_ROS1TR,AGH_Ktrunc}.

Hesselholt and Madsen \cite{HeMa04} studied the algebraic structure of topological restriction homology. On the level of homotopy groups, the maps $F$ and $R$ above induce  maps
\[
F\colon  \pi_q(\THH(A)^{C_{p^n}}) \to \pi_q(\THH(A)^{C_{p^{n-1}}}) \text{\quad and\quad} R\colon  \pi_q(\THH(A)^{C_{p^n}}) \to \pi_q(\THH(A)^{C_{p^{n-1}}}).
\]
The Frobenius map has an associated equivariant transfer map, called the Verschiebung:
\[
V\colon  \pi_q(\THH(A)^{C_{p^{n-1}}}) \to \pi_q(\THH(A)^{C_{p^{n}}}).
\]
For odd $p$, there is also a derivation 
\[
d\colon  \pi_q(\THH(A)^{C_{p^{n}}}) \to \pi_{q + 1}(S^1_+ \wedge \THH (A)^{C_{p^n}}) \to \pi_{q+1}(\THH(A)^{C_{p^{n}}}), 
\]
where the first map is exterior multiplication by $\sigma \in \pi_1^s (S^1_+)$, for $\sigma$ the element reducing to $(\id,0)\in \pi_1(S^1)\oplus\pi_1(S^0)$, and the second map is induced by the $S^1$-action. See \cite{HeMa04} and \cite[Lemma~1.4.2]{Hesselholt1996} for the argument that $d$ is in fact a derivation. 
Further, these operators satisfy various relations. For any commutative ring $A$, $FV$ is the multiplication by $p$ map, and for $A$ a $\mathbb{Z}_{(p)}$-algebra for an odd prime $p$, $FdV=d$.  These algebraic structures are the key in the connection between THH and Witt complexes, the latter of which is the subject of the next section.

\subsection{Witt vectors and Witt complexes}\label{introWitt}
In this section we give a brief introduction to Witt vectors and Witt complexes, and recall their connection to topological Hochschild homology. For a more thorough introduction, see \cite{HesselholtLecWitt} and \cite{HeMa97}.

\begin{defn} For a commutative ring $A$, the  \emph{length $k$ $p$-typical Witt vectors of $A$} is a ring $W_k(A)$ with underlying set 
\[
W_k(A) = \overset{k-1}{\underset{i=0}{\prod}} A = \{(a_0, a_1, \dots, a_{k-1} ) \mid a_i \in A\},
\]
and ring structure determined by the requirement that the map 
\begin{align*}
  w\colon W_k(A) &\to A^k \\
  (a_0,\dots,a_{k-1}) &\mapsto (w_0,\dots, w_{k-1})\quad\text{where}\quad w_n=\sum^n_{i=0} p^i a_i^{p^{n-i}}
\end{align*}
is a natural transformation of functors from rings to rings. The map $w$ is referred to as the \emph{ghost map}. The coordinates $(a_0, a_1, \dots, a_{k-1})$ are called the \emph{Witt coordinates}, while the coordinates $(w_0, w_1, \dots, w_{k-1})$ are the \emph{ghost coordinates}.
\end{defn}

The length $ k $ $p$-typical Witt vectors are related by several operators. 
There are two ring homomorphisms from the length $k$ $p$-typical Witt vectors to the length $k-1$ $p$-typical Witt vectors, the restriction and the Frobenius. 
\begin{defn}\label{restrictionWitt}
There is a ring homomorphism, called the \emph{restriction},
\[
R\colon W_k(A) \to W_{k-1}(A),
\]
given by projection onto the first $k-1$ Witt coordinates. 
\end{defn}
This restriction map makes $W_{\bullet}(A)$ a pro-ring. 

\begin{defn}
	There is a ring homomorphism, called the \emph{Frobenius},  
\[ F\colon W_k(A) \to W_{k-1}(A) \]
	determined by the formula
	\[
	w(F(a)) = (w_1(a), w_2(a), \dots, w_{k-1}(a)),
	\]
	where $a$ denotes the element with Witt coordinates $(a_0, a_1, \dots, a_{k-1}) \in W_k(A)$. 
\end{defn}

\begin{defn} 
	There is an additive map
	\[V\colon W_{k-1}(A) \to W_k (A)\]
	called the \emph{Verschiebung}. The map $V$ is specified in Witt coordinates by 
	\[V(a_0, a_1, \dots a_{k-1} ) = (0, a_0, a_1, \dots a_{k-1}).\] 
\end{defn}

There is also a multiplicative lift map, which is sometimes referred to as the Teichm\"uller map in the literature. 
\begin{defn}\label{defn:multlift}
The \emph{multiplicative lift map}
\[
[-]_k\colon A \to W_{k+1}(A)
\]
is given by $[a]_k = (a, 0, \dots, 0)$ in Witt coordinates. 
\end{defn}
\noindent Observe that the map $[-]_k$ lifts the $p^{k}$-power map in the sense that $F^{k}([a]_k) = a^{p^{k}}$.

\begin{rmk}\label{liftnames}
Note that we are using a non-standard naming convention for the multiplicative lift map above. This map from $A$ to $W_{k+1}(A)$ would typically be named $[-]_{k+1}$ in the Witt vector literature. We have chosen to name this map $[-]_{k}$ instead because this will align better with the natural naming in the equivariant case. See Remark \ref{liftnames2} for more discussion of this point.
\end{rmk}

These operators on the $p$-typical Witt vectors satisfy various relations. The Frobenius and Verschiebung maps both commute with the restriction. There is also a Frobenius reciprocity relation: $xV(y) = V(F(x)y)$. Further, the composite $FV$ is multiplication by $p$, where $p$ denotes the element $[1]$ added $p$-times.

Having recalled the definition of $p$-typical Witt vectors, we now discuss the close relationship between Witt vectors and topological Hochschild homology. Indeed, Hesselholt and Madsen show in \cite[Theorem 3.3]{HeMa97} that for a commutative ring $A$, 
\begin{equation}\label{eq:THHWitt}
\pi_0(\THH(A)^{C_{p^k}}) \cong W_{k+1}(A),
\end{equation} 
and this isomorphism is compatible with the maps $ F, V$, and $R $. 

To capture the algebraic structure of the homotopy groups of the fixed points of THH beyond $\pi_0$, Hesselholt and Madsen define the notion of a Witt complex. 

\begin{defn}[{\cite[p.~2]{HeMa04}}] \label{defn:Wittcplx_classical}
For $p$ an odd prime and $A$ a $\mathbb{Z}_{(p)}$-algebra, a \emph{Witt complex over $A$} consists of:
\begin{enumerate}
\item A pro-differential graded ring $E_\bullet^*$ and a strict map of pro-rings
\[
\lambda\colon W_{\bullet}(A) \to E_{\bullet}^0,
\]
where $W_{\bullet}(A)$ denotes the pro-ring of $p$-typical Witt vectors on $A$. 
\item A strict map of pro-graded rings
\[
F\colon E^*_{\bullet} \to E^*_{\bullet - 1}
\]
such that $\lambda F = F\lambda$ and such that for all $a \in A$,
\[
Fd\lambda([a]_k) = \lambda([a]_{k-1})^{p-1}d\lambda([a]_{k-1}),
\]
where $[a]_k = (a, 0, \ldots, 0) \in W_{k+1}(A)$ is the multiplicative representative.
\item A strict map of graded $E^*_{\bullet}$-modules
\[
V\colon F_*E^*_{\bullet-1} \to E^*_{\bullet},
\]
where $E^*_{\bullet-1}$ is considered as an $E^*_\bullet$-module via the map $F$, such that $\lambda V = V \lambda$, $FdV = d$ and $FV=p$. 
\end{enumerate}
\end{defn}

Hesselholt and Madsen prove that Witt complexes capture the algebraic structure of the fixed points of THH. 

\begin{thm}[\cite{HeMa04}]\label{classicalWittcxtheorem}
For $p$ an odd prime and $A$ a $\mathbb{Z}_{(p)}$-algebra, the pro-differential graded ring given by
\[
E^*_{\bullet} = \pi_*(\THH(A)^{C_{p^{\bullet-1}}}) = \TR_*^{\bullet}(A;p)
\] 
with $ \lambda $ the equivalence of (\ref{eq:THHWitt}) is a Witt complex. 
In this Witt complex, the maps $F$ and $V$ are the Frobenius and Verschiebung maps described in \Cref{introTHH}. The structure maps in the pro-system are the restriction maps $R$. The differential is the map $d$ from \Cref{introTHH}. 
\end{thm}

\begin{rmk}
If $p=2$ and $A$ is a $\mathbb{Z}_{(2)}$-algebra, the map $d$ defined in Section \ref{introTHH} does not square to zero. This necessitates a different definition for a 2-typical Witt complex. This was defined by Costeanu in \cite{Costeanu}. In the current work, we restrict to the case of odd primes. 
\end{rmk}

\section{Equivariant Hochschild theories}\label{equivariantHH}

In this section we recall the definitions of two equivariant Hochschild theories: twisted topological Hochschild homology and Hochschild homology for Green functors. We also recall necessary background from equivariant algebra to define the latter theory.  As the terminology suggests, the former theory is inherently topological while the latter is purely algebraic.  Nevertheless, a main result of \cite{BlGeHiLa} provides a close relationship between these theories which is central to our analysis of equivariant Witt complexes.

\subsection{Twisted topological Hochschild homology}

As noted in Section \ref{introTHH}, the topological Hochschild homology of a ring spectrum $R$, THH$(R)$, can be viewed as the equivariant norm $N_e^{S^1}R$. The norm perspective on THH lends itself well to generalizations. Indeed, for $C_n$ a finite cyclic group, and $R$ a genuine $C_n$-equivariant ring spectrum, Angeltveit, Blumberg, Gerhardt, Hill, Lawson, and Mandell \cite{AnBlGeHiLaMa} defined the $C_n$-twisted topological Hochschild homology of $R$, $\THH_{C_n}(R)$, as the norm $N_{C_n}^{S^1}(R)$. One can give an explicit construction of this norm in terms of a twisted cyclic bar construction, which we recall briefly here. More details of this construction can be found in \cite[Section 8]{AnBlGeHiLaMa}. 

\begin{defn}\label{defn_twisted_cyclic_bar}
 Let $R$ be an orthogonal $C_n$-ring spectrum indexed on the trivial universe $\mathbb{R}^{\infty}$. Let $g$ denote the generator $e^{2\pi i/n}$ of $C_n$. The $C_n$-\emph{twisted cyclic bar construction} on $R$, $N^{\cyc,C_n}_{\bullet}R$, is a simplicial $C_n$-spectrum, with $k$th level
    \[
N^{\cyc,C_n}_{k}R = R^{\wedge(k+1)}.
    \]
 In order to define the face and degeneracy maps, we define an operator $\alpha_k\colon R^{\wedge(k+1)} \to R^{\wedge(k+1)}$, which cyclically permutes the last factor to the front, and then acts on the new first factor by the generator $g$. Then the face and degeneracy maps of $N^{\cyc,C_n}_{\bullet}R$ are given by 
\[d_i = \begin{cases} \id^{\wedge i} \wedge \mu \wedge \id^{\wedge (k-i-1)} & \text{for } 0\leq i<k \\  (\mu \wedge \id^{\wedge (k-1)}) \circ \alpha_k& \text{for }i=k,
\end{cases}
\] 
and 
\[
s_i =  \id^{\wedge (i+1)} \wedge \eta \wedge \id^{\wedge (k-i)} \text{\quad for }  0\leq i \leq k.
\]
Here $\mu$ and $\eta$ denote the multiplication and unit maps of $R$, respectively.  
\end{defn}

\begin{rmk}
    As shown in \cite[\S8]{AnBlGeHiLaMa}, the twisted cyclic bar construction is naturally a $ \Lambda^\op_n $-object in the sense of B\"okstedt--Hsiang--Madsen \cite[Definition 1.5]{BHM}. Hence the geometric realization of the twisted cyclic bar construction inherits an action of $ S^1 $. 
\end{rmk}
This twisted cyclic bar construction gives a model for the equivariant norm. Therefore, we have the following definition of twisted topological Hochschild homology. 
\begin{defn}[{\cite[\S8]{{AnBlGeHiLaMa}}}]\label{defn:twisted_THH}
    For $R$ an orthogonal $C_n$-ring spectrum, the \emph{$ C_n $-twisted topological Hochschild homology} $ \THH_{C_n}(R) $ of $R$ is defined to be
    \[
    \THH_{C_n}(R) = N_{C_n}^{S^1}(R) = \mathcal{I}_{\mathbb{R}^{\infty}}^U |N^{\cyc,C_n}_{\bullet}(\mathcal{I}_{\widetilde{U}}^{\mathbb{R}^{\infty}}R)|.
    \]
  Here $U$ is a complete $S^1$-universe, $\widetilde{U} = i_{C_n}^* U$ is the pullback of $U$ to $C_n$, and $\mathcal{I}$ denotes a change of universe functor.  
\end{defn}

For a $C_n$-Green functor $\m{R}$, the Eilenberg--MacLane spectrum $H\m{R}$ is a $C_n$-ring spectrum, as shown in \cite{ullman2013tambara}. We write $\THH_{C_n}(\m{R})$ for the twisted topological Hochschild homology of $H\m{R}$.  When $n$ is relatively prime to $p$, twisted topological Hochschild homology exhibits a structure related to $p$-cyclotomicity.

\begin{thm}\label{twistedpcyclotomicTHH}
Let $R$ be an orthogonal $C_n$-ring spectrum and let $p$ be a prime.  If $p$ does not divide $n$, then the spectrum $\THH_{C_n}(R)$ is twisted $p$-cyclotomic.  In particular, there is a map
\[ \THH_{C_n}(R)\to \rho_p^*\Phi^{C_p} \THH_{C_n} (\phi_p^* R)\]
that is an equivalence; the inverse of this map is the twisted $p$-cyclotomic structure map.  Here $\phi_p\colon C_n\to C_n$ is the $p$th power map. Since $n$ is relatively prime to $p$, $\phi_p$ has finite order $\nu$.  Hence iterating this map provides an $S^1$-equivalence
\[ \THH_{C_n}(R)\to \rho^*_{p^\nu}\Phi^{C_{p^\nu}}\THH_{C_n}(R).\]
\end{thm}

\begin{rmk}
It is claimed in \cite[Theorem 8.6]{AnBlGeHiLaMa} that twisted THH is $p$-cyclotomic when $p$ does not divide $n$. It was pointed out by John Rognes that the $p$-cyclotomic structure map in \cite{AnBlGeHiLaMa} is missing a twist by the $p$th power map. This can also be seen by comparing to the work of Krause--McCandless--Nikolaus on polygonic spectra in \cite{KMN}. The above statement corrects this error and describes the resulting twisted $p$-cyclotomic structure.  
\end{rmk}

\begin{rmk}\label{rmk:restrictionfortwistedpcycloTHH} When $n$ is relatively prime to $p$, this twisted $p$-cyclotomic structure induces maps $R$ defined as the composite
\[
R\colon \THH_{C_n}(R)^{C_{p^k}} =  (\rho_{p^\nu}^*\THH_{C_n}(R)^{C_{p^\nu}})^{C_{p^{k-\nu}}} \to (\rho_{p^\nu}^*\Phi^{C_{p^\nu}} \THH_{C_n}(R))^{C_{p^{k-\nu }}} \xrightarrow{\simeq} \THH_{C_n}(R)^{C_{p^{k-\nu}}}.
\]
\end{rmk}

Twisted topological Hochschild homology of equivariant spectra has an algebraic analogue, Hochschild homology for Green functors, which we recall in \Cref{subsect:HHGreen}. 
The relationship between twisted $ \THH $ and $ \HH $ for Green functors is analogous to the relationship between classical $ \THH $ and classical $ \HH $ described in Remark \ref{rmk:THH_HH_relationship_classical}. 
We will be able to make this relationship precise after recalling some foundations of equivariant algebra and the definition of Hochschild homology for Green functors, which occupy the remainder of this section. 

\subsection{Equivariant algebra} 
In this subsection we briefly recall the foundations of equivariant algebra, introducing some of the algebraic structures that arise in the study of equivariant homotopy theory. We begin by defining \emph{Mackey functors}, the equivariant analogue of abelian groups. 
\begin{defn} Let $G$ be a finite group.  Let $\mathcal{O}_G$ denote the \emph{orbit category} of $G$, that is, the category whose objects are finite sets with transitive $G$-action and whose morphisms are $G$-equivariant maps.  Let $\mathrm{Fin}_G$ denote the category of finite $G$-sets; note that $\mathrm{Fin}_G$ is the coproduct completion of $\mathcal{O}_G$.
\end{defn}
\begin{rmk}\label{rmk:Weyl_group}
    Let $ H \subset G $ be a subgroup, and suppose $ g \in G $ normalizes $ H $, in that $ g H g^{-1} = H $. 
    Then multiplication by $g$ induces an action on the set of cosets $ g \colon G/H \xto{\cong} G/H $. 
    In fact, the group of automorphisms of the object $ G/H $ in $ \mathcal{O}_G $ is isomorphic to the \emph{Weyl group of $ H $ in $ G $}, $ W_G(H) =  (N_GH)/H$. Here $N_GH$ denotes the normalizer of $ H $ in $ G $.  
\end{rmk}
\begin{defn}  Let $\Span(\Fin_G)$ denote the \emph{span category} of $\Fin_G$, that is, the category whose objects are those of $\Fin_G$ but where morphisms from $X$ to $Y$  are isomorphism classes of ``spans'' $X\leftarrow A\to Y$ of maps in $\Fin_G$.  Composition in $\Span(\Fin_G)$ is given by pullback.  The categorical product in $\Span(\Fin_G)$ is given on objects by disjoint union.
\end{defn}

\begin{rmk}
    The set of morphisms from $X$ to $Y$ in $\Span(\Fin_G)$ has a natural commutative monoid structure given by disjoint union on the objects that are the source of both maps in the spans.  
    The usual Burnside category $\mathcal{B}_G$  of $G$ is obtained from $\Span(\Fin_G)$ by group completing morphism sets.
\end{rmk}

\begin{defn}\label{defn:Mackeyfunctors} Let $G$ be finite group.  A \emph{$G$-Mackey functor} $\m{M}$ is a product-preserving functor $\Span(\Fin_G)\to \Ab$.  The category of $G$-Mackey functors is denoted $\Mack_G$. 
\end{defn}
A $G$-Mackey functor also gives rise to a Mackey functor for any subgroup $H$ of $G$.
\begin{defn}\label{defn:restrictionMackey} Let $S$ be a finite $H$-set.
The assignment $ S \mapsto G \times_H S $ induces functors $ \mathcal{O}_H \to \mathcal{O}_G $ and $ \Span(\Fin_H) \to \Span(\Fin_G) $. 
Precomposition with the latter gives a functor $ i^*_H \colon \Mack_G \to \Mack_H $. 
If $ \m{M} $ is a $ G $-Mackey functor, we call $ i^*_H(\m{M}) $ the \emph{restriction of $ \m{M} $ to $H$}. 
\end{defn}

\begin{rmk}
We could alternately define Mackey functors as follows.  A semi-Mackey functor is a product-preserving functor $\Span(\Fin_G)\to \Set$.  If $F$ is such a product-preserving functor, then for each $G$-set $X$, there is a commutative monoid structure on $F(X)$.  A Mackey functor is thus defined to be a semi-Mackey functor for which all of these monoids are group-complete. See \cite[Section 6]{Tambara_multitransfer} as well as \cite{BrunWittTambara,strickland2012tambara} for more on completion in this context.
\end{rmk}

The data of a Mackey functor $\m{M}$ has the following concrete description, which was Dress's original perspective \cite{Dress73}.  A Mackey functor $\m{M}$ consists of a pair of functors
\[
M^*, M_* \colon \mathcal{O}_G \to \Ab, 
\]
 such that $M^*$ is contravariant, $M_*$ is covariant, the values of $M^*$ and $M_*$ agree on every orbit $G/H$ in $\mathcal{O}_G$, and a double coset formula relates $M_*$ and $M^*$ on morphisms.  We write $\m{M}(G/H)$ for the shared value of $M_*$ and $M^*$ on the orbit $G/H$. Given a map $f\colon G/H\to G/K$ in $\mathcal{O}_G$, we call $M^*(f)$ the \emph{restriction} along $f$ and $M_*(f)$ the \emph{transfer} along $f$.  When $H$ is a subgroup of $K$ and $f\colon G/H\to G/K$ is the canonical quotient map, we write $M^*(f)=\res_H^K$ and $M_*(f)=\tr_H^K$. By Remark \ref{rmk:Weyl_group}, for each subgroup $ H \subset G $ the Weyl group $ W_G(H) $ acts on $ \m{M}(G/H) $. 
 For an accessible comparison of various descriptions of Mackey functors, see \cite[Chapter 1]{Bouc_GreenGSets}.

\begin{exmp}\label{ex:Burnside}
The \emph{Burnside ring} of a group $G$, denoted $A(G)$, is the group completion of the monoid of isomorphism classes of finite $G$-sets under disjoint union. The \emph{Burnside Mackey functor} for $G$, denoted, $\m{A}_G$, is defined by 
\[
\m{A}_G(G/H) = A(H).
\]
Let $H \subset K \subset G$.  Let $Z$ be a finite $H$-set and let $Y$  be a finite $K$-set. The transfer maps of $\m{A}_G$ are defined by $\tr_H^K([Z]) = [ K \times_H Z]$, and restriction maps are defined by $\res_H^K([Y]) = [i_H^K(Y)].$ Here $i_H^K$ denotes the restriction functor from finite $K$-sets to finite $H$-sets.   
\end{exmp}

Mackey functors are prominent in equivariant stable homotopy theory because they are the algebraic structure present on equivariant homotopy groups.  
For a genuine $G$-spectrum $X$, the equivariant homotopy groups of $X$ yield a $G$-Mackey functor $\m{\pi}_n^G(X)$ for each $n\geq0$, defined by 
\[ \m{\pi}_n^G(X)(G/H)=\pi_n(X^H)\,.\]

\begin{exmp}  
For any $G$-Mackey functor $\m{M}$, there is an Eilenberg--MacLane $G$-spectrum $H\m{M}$ that is defined by the property 
\[ \m{\pi}_n^G(H\m{M})=\begin{cases} \m{M} & n=0\\ 0 & n\neq 0.\end{cases}\]
\end{exmp}

The category $\Mack_G$ has a symmetric monoidal structure that arises as follows.  

\begin{defn}\label{rec:eqvtsmashproduct} %
 The category $\Span(\Fin_{G})$ inherits a symmetric monoidal structure from $ \mathrm{Fin}_{G} $ given on underlying objects by cartesian product of finite $ G $-sets.  The box product $\m{M}\square\m{N}$ of $G$-Mackey functors $\m{M}$ and $\m{N}$ is defined to be the left Kan extension in the diagram
\[
\begin{tikzcd}\Span(\Fin_G)\times \Span(\Fin_G)\ar[r,"{\m{M}\times\m{N}}"]\ar[d, "-\times-", swap]_{\times} & \Ab\times \Ab\ar[r,"{\otimes}"] &\Ab \\
\Span(\Fin_G)\ar[urr,"{\m{M}\square\m{N}}", swap,bend right=10]
\end{tikzcd}
\]
This construction is also known as \emph{Day convolution}, as described in \cite{BrunWittTambara}.  Its interaction with group completion is discussed in \cite[Appendix A]{strickland2012tambara}. 
\end{defn}
The Burnside Mackey functor for $G$, recalled in Definition \ref{ex:Burnside}, is the unit for the box product. There is also a graded version of the box product, for graded Mackey functors $\m{M}_*$ and $\m{N}_*$. 
\begin{defn}
For graded Mackey functors $\m{M}_*$ and $\m{N}_*,$ the \emph{graded box product} $\m{M}_* \square \m{N}_*$ is defined by 
\[
(\m{M}_* \square \m{N}_*)_k = \bigoplus_{i+j=k} \m{M}_i \square \m{N}_j.
\]
This gives the category of graded Mackey functors a symmetric monoidal structure.
\end{defn}

\begin{defn}
    Let $ G $ be a finite group. 
    A \emph{Green functor} for the group $ G $ is a monoid with respect to the box product in the category of $G$-Mackey functors. 
    We denote the category of $ G $-Green functors by $ \Green_G $.  Note in particular that a Green functor enjoys a ring structure at every level $\m{M}(A)$.
    A \emph{graded Green functor} for the group $ G $ is a monoid in the category of graded $G$-Mackey functors under the graded box product. Similarly, graded Green functors enjoy levelwise graded ring structures.
\end{defn}

\begin{rmk}
A $ G $-spectrum $ X $ is said to be \emph{connective} if its fixed points $ X^H $ are connective for all subgroups $ H \subset G $. 
   The functor $ \underline{\pi}_0 \colon \Spectra^G \to \Mack_G $ is symmetric monoidal when restricted to connective $ G $-spectra. 
   Therefore, the Eilenberg--MacLane spectrum functor $ H $ is lax symmetric monoidal and for $\m{R}$ a $G$-Green functor, $H\m{R}$ is a $G$-equivariant ring spectrum.
\end{rmk}

While commutative Green functors are the commutative ring objects in Mackey functors, they do not capture the full strength of equivariant commutativity.  This instead is encoded in the notion of a Tambara functor.  

\begin{defn}[{\cite{Tambara_multitransfer}}]\label{rec:Tambara_functor}  
Let $G$ be a finite group. A $G$-\emph{Tambara functor} is defined as a product preserving functor from a category of ``bispans'' to $\Set $, where a ``bispan'' or ``polynomial diagram'' is a diagram
\[ X\leftarrow A\to B\to Y\]
of $G$-sets.  A complete definition of this category can be found in \cite[Section 2.1]{MR3773736} or \cite[Section 6]{strickland2012tambara}; see also \cite{Tambara_multitransfer, BrunWittTambara}.

Restricting to those bispans so that $ A = B $, we see that a $ G $-Tambara functor has an underlying Mackey functor. 
The bispans given by fold  maps $ A^{\sqcup 2} \xleftarrow{=}A^{\sqcup 2} \xrightarrow{\nabla} A\xto{=}A $ are sent to multiplication maps making each $ \m{M}(A) $ into a ring---in particular a $ G $-Tambara functor has an underlying $ G $-Green functor. 
What sets a Tambara functor $ \m{M} $ apart from a Green functor is the additional data of \emph{norm maps} 
\begin{equation*}
    n_K^H \colon \m{M}(G/K)\to \m{M}(G/H) 
\end{equation*} 
for every canonical quotient map $ f \colon G/K\to G/H $ of finite $G$-sets.  This map is induced by the bispan $G/K \xleftarrow{=} G/K\xto{f} G/H \xto{=} G/H$. 
The maps $ n_K^H$ are multiplicative in that they preserve the underlying multiplicative monoid structure on the rings $\m{M}(G/K)$ and $\m{M}(G/H)$, but they are not ring maps. 
The category of all $ G $-Tambara functors is denoted by $ \Tamb_G $. 
Similarly to the case of Mackey functors, an inclusion of a subgroup $ H \subset G $ induces a functor on categories of bispans, which induces a \emph{restriction} functor $ i_H \colon \Tamb_G \to \Tamb_H $. 
For details, see \cite{BrunWittTambara}.  
\end{defn}

In \cite{BlGeHiLa}, the authors define a notion of geometric fixed points for Mackey functors.  Recall that a \emph{family} of subgroups of $ G $ is a collection of subgroups of $ G $ which is closed under subconjugacy, i.e. for a family $\mathcal{F}$, if $H \in \mathcal{F}$ and $g^{-1}Kg \subset H$, then $K \in \mathcal{F}$.  

\begin{defn}[{\cite[Definitions 5.4, 5.5]{BlGeHiLa}}]
For a finite group $G$ and a normal subgroup $N$ of $G$, let $\mathcal{F}_N$ denote the family of subgroups of $G$ which do not contain $N$. Let $E\mathcal{F}_N(\m{A}_G)$ denote the sub-Mackey functor of the Burnside Mackey functor $\m{A}_G$ which is generated by $\m{A}_G(G/H)$ for all subgroups $H \in \mathcal{F}_N$. Further, let 
\[
\widetilde{E}\mathcal{F}_N(\m{A}_G) = \m{A}_G/E\mathcal{F}_N(\m{A}_G). 
\]
If $\m{M}$ is a $G$-Mackey functor, let
\[
\widetilde{E}\mathcal{F}_N(\m{M}) = \m{M} \square \widetilde{E}\mathcal{F}_N(\m{A}_G). 
\]
\end{defn}
\begin{notation}\label{ntn:fixed_point_as_Mackey_for_quotient}
    Let $ N $ be a normal subgroup of $ G $. 
    We may regard any $ G/N $-set as a $ G $-set along the canonical projection $ G \to G/N $. This assignment promotes to a product-preserving functor $\Span(\Fin_{G/N}) \to \Span(\Fin_G)$. 
   Pullback along this functor induces $ \zeta_N \colon \Mack_G \to \Mack_{G/N} $. This functor is denoted $q_*$ in \cite{HiMeQu}, where it is shown to preserve Green functors and Tambara functors. 
\end{notation}

\begin{defn}[{\cite[Definition 5.10]{BlGeHiLa}}]
For a normal subgroup $N \subset G$, the $N$-\emph{geometric fixed points} of a $G$-Mackey functor $\m{M}$ is the $G/N$-Mackey functor defined by
\[
\Phi^N(\m{M}) = \zeta_N(\widetilde{E}\mathcal{F}_N(\m{M})).
\]
In \cite[Proposition 5.13]{BlGeHiLa}, the authors show that $\Phi^N$ is strong symmetric monoidal and hence lifts to a functor from $G$-Green functors to $G/N$-Green functors. 
\end{defn}

\begin{rmk}\label{eq:cat_to_geom_fixpt}There is a natural transformation of functors $ \Mack_G \to \Mack_{G/N} $,
    \begin{equation*}
        \zeta_N \to \Phi^{N},
    \end{equation*}
induced by the natural map 
\[
\m{M} \to \widetilde{E}\mathcal{F}_N(\m{M}).
\]    
As both $\Phi^N$ and $\zeta_N$ take $ G $-Green functors to $ G/N $-Green functors, the natural transformation in Remark \ref{eq:cat_to_geom_fixpt} yields a natural transformation of functors $\Green_G \to  \Green_{G/N}$.
\end{rmk}

For a $G$-Mackey functor $\m{M}$, and $N \subset G$ a normal subgroup, one can think of $\zeta_N \m{M}$ as the $N$-fixed points of $\m{M}$, and the natural transformation above as analogous to the usual map from fixed points to geometric fixed points in $G$-spectra. Indeed, in \cite{BlGeHiLa} and \cite{AKGH} the authors show that these notions of fixed points for Mackey functors are consistent with the notions in $ G $-spectra. In particular, they show that 
\[
 \Phi^N \m{M} \cong \m{\pi}_0^{G/N}(\Phi^NH\m{M}) 
\]
and
\[
\zeta_N(\m{M}) \cong \m{\pi}_0^{G/N}((H\m{M})^N). 
\]

We now recall the definition of the norm construction on a Mackey functor.

\begin{defn}[{\cite[Definition 5.9]{HillHopkins}}] \label{defn:Mackey_functor_norm}
    Let $G$ be a finite group, $H$ a subgroup of $G$, and $\m{M}$ an $H$-Mackey functor. The Mackey functor \emph{norm} $N_H^G\m{M}$ is the $G$-Mackey functor
    \[
    N_H^G\m{M}:= \m{\pi}_0N_H^G(H\m{M})
    \] 
    where $ N_H^G $ on the right-hand side of the definition is the Hill--Hopkins--Ravenel norm in spectra from \cite{HHR}.
\end{defn}
While the definition above highlights the relationships between norms in Mackey functors and norms in spectra, work of Mazur \cite{MThesis} and Hoyer \cite{Hoy} gives purely algebraic constructions of Mackey functor norms. Mazur's concrete construction applies to norms for cyclic $p$-groups, and Hoyer generalizes this to a description of the norm for any finite group $G$. 
Mackey functor norms also appear under a different guise in earlier work of Bouc \cite{Bouc}. 

The norm $ N \colon \Mack_H \to \Mack_{G} $ can be thought of as a multiplicative induction, or a tensor/box product over the cosets $ G/H $; this is made precise by the following result of Hoyer \cite[Corollary~2.6.1]{Hoy}: 
\begin{prop} \label{prop:norm_as_indexed_product}

Let $G$ be a finite abelian group. Let $H, \, K \subset G$ be subgroups and let $\underline{M}$ be an $H$-Mackey functor. Then there is a natural isomorphism of $ K$-Mackey functors
  \[ i_K^* N_H^G \underline{M} \cong \bigboxvoid_{\left|G/HK\right|} N_{K \cap H}^K (i_{K \cap H}^* \underline{M}). \] 
\end{prop}
In particular, for a $C_n$-Mackey functor $\m{M}$,
\begin{align*}
  N_{C_n}^{C_{n p^k}} \underline{M} (C_{n p^k} / e)  & =(i_{C_n}^*N_{C_n}^{C_{n p^k}} \underline{M}) (C_{n} / e)= \underline{M}^{\square p^k} (C_n / e) \\
  N_{C_n}^{C_{n p^k}} \underline{M} (C_{n p^k} / C_n)  & =(i_{C_n}^*N_{C_n}^{C_{n p^k}} \underline{M})(C_n/C_n)= \underline{M}^{\square p^k} (C_n / C_n) \\
  N_{C_n}^{C_{n p^k}} \underline{M} (C_{n p^k} / C_{p^q})  &  =(i_{C_{p^q}}^*N_{C_n}^{C_{n p^k}}
  \underline{M})(C_{p^q}/C_{{p^q}}) = \left(N_{e}^{C_{p^q}} i_e \underline{M} \right)^{\square \left| C_{n p^k} / C_{n p^q} \right|} (C_{p^q}/C_{p^q})
\end{align*}
where $0 < q \leq k$ and $n$ is relatively prime to $p$.

As for equivariant spectra, there is also a nice relationship between geometric fixed points and norms for Mackey functors. 

\begin{prop}[{\cite[Theorem 5.15]{BlGeHiLa}}] \label{propitem:norm_geom_fixpt}  Let $ N $ be a normal subgroup of $ G $ and let $ \m{M} $ be an $ H $-Mackey functor, for $H \subset G$. 
  Then there is a natural isomorphism of functors $ \Mack_H \to \Mack_{G/N} $
  \begin{equation*}
      \Phi^{N} \left(N_H^G\m{M}\right) \simeq N_{HN/N}^{G/N} \Phi^{H\cap N}\m{M}\,.
  \end{equation*}
\end{prop}

Finally, we recall a result about norms of Tambara functors.

\begin{prop}[{\cite[Proposition 6.9]{MR3773736} or \cite[\S2.3.4]{Hoy}}]\label{prop:Tambara_norm_restriction_adjunction}
    Let $G$ be a finite group and $H$ a subgroup of $G$. 
    Then the functor $ N^G_H $ lifts to a functor $ \Tamb_H \to \Tamb_G $ with right adjoint $i^*_H$.  
\end{prop}

\subsection{Hochschild homology for Green functors} \label{subsect:HHGreen}
We recall the definition of Hochschild homology for Green functors from \cite{BlGeHiLa}. Let $ C_n $ denote the cyclic group of order $ n $ and let $ g \in C_n $ denote the generator $g=e^{2\pi i/n}$. 

\begin{defn} [{\cite[Definitions 2.20, 2.25]{BlGeHiLa}}] 
    Let $ \underline{R} $ be a $ C_n $-Green functor. 
    The \emph{$C_n$-twisted cyclic nerve of} $ \underline{R} $, denoted $\m{HC}^{C_n}_{\bullet}(\m{R})$, is a simplicial $ C_n $-Mackey functor with $ k $-simplices
    \begin{equation*}
        [k] \mapsto \underline{R}^{\square (k+1)}.
    \end{equation*}
    In order to define the face and degeneracy maps, we define an operator $\alpha_k\colon \m{R}^{\square(k+1)} \to \m{R}^{\square(k+1)}$, which cyclically permutes the last factor to the front, and then acts on the new first factor by the generator $g$. Then the face and degeneracy maps of $\m{HC}^{C_n}_{\bullet}(\m{R})$ are given by 
\[d_i = \begin{cases} \id^{\square i} \square \mu \square \id^{\square (k-i-1)} & \text{for } 0\leq i<k \\  (\mu \square \id^{\square (k-1)}) \circ \alpha_k& \text{for }i=k,
\end{cases}
\] 
and 
\[
s_i =  \id^{\square (i+1)} \square \eta \square \id^{\square (k-i)} \text{\quad for }  0\leq i \leq k.
\]
Here $\mu$ and $\eta$ denote the multiplication and unit maps of $\m{R}$, respectively.

    The \emph{$C_n$-twisted Hochschild homology} of $\m{R}$ is the homology 
    \[
    \m{\HH}_i^{C_n}(\m{R}) = H_i(\m{HC}^{C_n}_{\bullet}(\m{R})) \,
    \]
    of the $ C_n $-twisted cyclic nerve of $\m{R}$. 
\end{defn}
Here we use the equivalence between the category of simplicial Mackey functors and the category of non-negatively graded dg-Mackey functors, given by applying the Dold--Kan correspondence at each orbit. By definition, the homology of a simplicial Mackey functor is the homology of the associated normalized dg-Mackey functor (see \cite[\S4]{BlGeHiLa}).

The definition above allows us to take the $G$-twisted Hochschild homology of a $G$-Green functor, whenever $G \subset S^1$ is a finite subgroup. More generally, we will need the notion of $G$-twisted Hochschild homology for an $H$-Green functor, where $H\subset G \subset S^1$ are finite subgroups. The $G$-twisted Hochschild homology of an $H$-Green functor is defined by taking the $G$-twisted Hochschild homology of Mackey functor norms, as introduced in  \Cref{defn:Mackey_functor_norm}.
\begin{defn}[{\cite[Definition 3.19]{BlGeHiLa}}]
Let $H\subset G\subset S^1$ be finite subgroups, and let $\m{R}$ be a Green functor for $H$. The \emph{$G$-twisted cyclic nerve of $\m{R}$} is
\[
\m{HC}_H^G(\m{R})_{\bullet}:= \m{HC}^G_{\bullet}(N_H^G(\m{R})).
\]

The \emph{$G$-twisted Hochschild homology of $\m{R}$} is
\[
\m{\HH}_H^G(\m{R})_i := H_i(\m{HC}_H^R(\m{R})_{\bullet}).
\]
\end{defn}
This Hochschild homology for Green functors is an algebraic analogue of twisted topological Hochschild homology. 
Indeed, in \cite[Theorem 5.1]{BlGeHiLa} the authors prove that for $H \subset G \subset S^1$ finite subgroups, and $R$ a connective, commutative $H$-ring spectrum, there is a natural map
\begin{equation*} \label{eq:BGHLnatMap}
\m{\pi}_k^G \THH_H(R) \to \m{\HH}_H^G(\m{\pi}_0^H(R))_k,
\end{equation*}
which is an isomorphism when $k=0$. 
\begin{rmk}\label{rmk:norm_to_twisted_HH_quotient} 
    The natural map from the zeroth level of a connected chain complex to zeroth homology induces a canonical map  
    \[
    q \colon N_{C_n}^{C_{p^kn}}\m{R} \to \m{\HH}_{C_n}^{C_{p^kn}}(\m{R})_0 = \m{\mathbb{W}}_{C_{p^kn}}(\m{R}) \,.
    \]
    In fact, $ q = \m{\pi}_0(\eta_{S^1}) $ where $ \eta_{S^1}\colon N_{C_n}^{C_{p^kn}}H\m{R} \to i^*_{C_{p^kn}}N^{S^1}_{C_{p^kn}} N_{C_n}^{C_{p^kn}}H\m{R} \simeq i^*_{C_{p^kn}}\THH_{C_n}(\m{R}) $ is the unit of the $ (i^*_{C_{p^kn}},N^{S^1}_{C_{p^kn}}) $-adjunction. 
\end{rmk}

\section{Equivariant Witt vectors}\label{sec:EquivariantWitt}
As recalled in Section \ref{sec:background} above, for a ring $A$, the zeroth homotopy group of the fixed points of $\THH(A)$ with respect to the subgroup $ C_{p^k} \subset S^1 $ is isomorphic to the length $k+1$ $ p $-typical Witt vectors on $ A $ \cite[\S3.3]{HeMa97}.  
In \cite{BlGeHiLa}, Blumberg, Gerhardt, Hill, and Lawson define Witt vectors for Green functors for cyclic groups via analogy with the aforementioned result. 
In this section we recall the definitions from that paper, and construct new operators on these equivariant Witt vectors. 
\begin{defn}[{\cite[Definition 1.4]{BlGeHiLa}}]\label{defn:equivariant_Witt_vectors}
  	Let $H \subset G \subset S^1$ be finite subgroups. 
    The \emph{$ G $-Witt vectors} of an $H$-Green functor $\m{R}$ are defined by 
    \begin{equation*}
    \m{\mathbb{W}}_G(\m{R}) \coloneqq \m{\HH}_H^G(\m{R})_0.
    \end{equation*}
    Let $\m{R}$ be a Green functor for $C_n \subset S^1$. 
    Then the equivariant Witt vectors $\m{\mathbb{W}}_{C_{p^kn}}(\m{R})$ are referred to as the \emph{length} $k+1$ $p$-\emph{typical Witt vectors} of $\m{R}$. 
\end{defn}

\begin{rmk}\label{Wittcommutative}
When $\m{R}$ is a commutative $H$-Green functor, the equivariant Witt vectors $\m{\mathbb{W}}_G(\m{R})$ form a commutative Green functor by \cite[Proposition 2.26]{BlGeHiLa}.
\end{rmk}

\begin{rmk}\label{rmk:Witt_vectors_as_norm_quotient}
    Let $ \m{R} $ be a commutative $ C_n $-Green functor. 
    Then by the discussion immediately following \cite[Proposition 6.4]{BlGeHiLa}, there is an isomorphism $ \m{\mathbb{W}}_{C_{p^kn}}(\m{R}) \cong \left(N_{C_n}^{C_{p^kn}} \m{R} \right)_{C_{p^kn}} $, where the latter denotes the strict quotient with respect to the $ C_{p^kn} $-action. In particular, the map $q$ of \Cref{rmk:norm_to_twisted_HH_quotient} is the map $q\colon N_{C_n}^{C_{p^kn}}\m{R}\to \m{\mathbb{W}}_{C_{p^kn}}(\m{R})$ to this strict quotient.
\end{rmk}
\begin{exmp} \cite[Proposition 6.1]{BlGeHiLa}\label{ex:Wittvec_as_e_eqvt}
    Let $R$ be a commutative ring. We consider this as a commutative $H$-Green functor for $ H = \{e\}$, the trivial group. Let $ G = C_{p^k} $. Then \Cref{defn:equivariant_Witt_vectors} recovers the classical length $ k+1 $ $ p$-typical Witt vectors in the sense that
    \[
    \m{\mathbb{W}}_{C_{p^k}}(R)(C_{p^k}/C_{p^k}) \cong W_{k+1}(R).
    \] 
\end{exmp}
\begin{exmp}\label{Burnside_eqvt_Witt_vectors}
	Let $H \subset G \subset S^1$ be finite subgroups and $\m{A}_H$ be the Burnside Mackey functor for $H.$ Then
	$$ \m{\mathbb{W}}_G (\m{A}_H) \cong \m{A}_G.$$
The equivalence follows from the fact that $N_H^G \m{A}_H \cong \m{A}_G$ and the Weyl group acts as the identity on isomorphism classes of spans. 
\end{exmp}

More generally, we have the following.
\begin{exmp}\label{trivial_weyl_Witt}
	Let $H \subset G \subset S^1$ be finite subgroups and suppose $\m{R}$ is an $H$-Green functor such that $N_H^G \m{R}$ is a commutative Green functor with trivial Weyl action. Then 
	\[ \m{\mathbb{W}}_G(\m{R}) = N_H^G \m{R}.\] 
\end{exmp}

In Section \ref{introWitt}, we recalled the definitions of the operators $F, V, R,$ and $[-]_k$ for classical $p$-typical Witt vectors. We now consider the analogous operators for the $p$-typical Witt vectors in the equivariant case. 

For $\m{R}$ a $C_n$-Green functor and $m$ dividing $n$, the map 
\[
F\colon \m{\mathbb{W}}_{C_{p^kn}}(\m{R})(C_{p^kn}/C_{p^jm}) \to \m{\mathbb{W}}_{C_{p^kn}}(\m{R})(C_{p^kn}/C_{p^{j-1}m})\]
is given by the restriction map $\res_{C_{p^{j-1}m}}^{C_{p^jm}}$ of the $C_{p^kn}$-Mackey functor $\m{\mathbb{W}}_{C_{p^kn}}(\m{R})$. Similarly, the map 
\[
V\colon \m{\mathbb{W}}_{C_{p^kn}}(\m{R})(C_{p^kn}/C_{p^{j-1}m}) \to \m{\mathbb{W}}_{C_{p^kn}}(\m{R})(C_{p^kn}/C_{p^{j}m})
\]
is given by the transfer map $\tr_{C_{p^{j-1}m}}^{C_{p^jm}}$.

We now define the restriction map $r$ on the $p$-typical Witt vectors of $\m{R}$. Let $\phi_p\colon C_n \to C_n$ denote the $p$th power map.  We recall from \cite[Proposition 5.17]{BlGeHiLa} that there is an isomorphism of simplicial Green functors 
\begin{equation}\label{algebraic_cyclotomic}
\Phi^{C_p}(\m{HC}_{C_n}^{C_{p^kn}}(\phi_p^*\m{R})_{\bullet}) \cong \m{HC}_{C_{pn}/C_p}^{C_{p^kn}/C_p}(\phi_p^*\m{R})_{\bullet}.
\end{equation}
Observe that the right-hand side can be identified as 
\[
\m{HC}_{C_{pn}/C_p}^{C_{p^kn}/C_p}(\phi_p^*\m{R})_{\bullet} \cong \m{HC}_{C_{n}}^{C_{p^{k-1}n}}(\m{R})_{\bullet}.
\]
We then take the zeroth homology of the associated dg-Green functors in \ref{algebraic_cyclotomic}. This yields an isomorphism
\begin{equation}\label{WittTwistedCyclotomic}
\Phi^{C_p}\m{\W}_{C_{p^kn}}(\phi_p^*\m{R}) \cong \m{\W}_{C_{p^{k-1}n}}(\m{R}).
\end{equation} 
Recall that since $n$ is prime to $p$, $\phi_p$ has a finite order $\nu$. Hence by iteration of \ref{WittTwistedCyclotomic} we get an isomorphism 
\begin{equation}\label{WittCyclotomic}
\Phi^{C_{p^{\nu}}}\m{\W}_{C_{p^kn}}(\m{R}) \cong \m{\W}_{C_{p^{k-\nu}n}}(\m{R}).
\end{equation}

\begin{defn}\label{defn:eqvt_Witt_restriction_maps} 
 For $\m{R}$ a $C_n$-Green functor, define a \emph{restriction map} on the $p$-typical Witt vectors of $\m{R}$
 \[
r \colon \zeta_{C_{p^{\nu}}}\m{\W}_{C_{p^kn}}(\m{R}) \to \m{\W}_{C_{p^{k-\nu}n}}(\m{R})
 \]
  by taking the canonical morphism of $C_{p^{k-\nu}n}$-Green functors
    \begin{equation*}
       \zeta_{C_{p^{\nu}}}\m{\W}_{C_{p^kn}}(\m{R}) \to \Phi^{C_{p^{\nu}}}\m{\W}_{C_{p^kn}}(\m{R})
    \end{equation*}
    of  Remark \ref{eq:cat_to_geom_fixpt} and composing it with the isomorphism of Equation \ref{WittCyclotomic}.
\end{defn}

\begin{rmk}\label{rmk:restrictionmapsWitt}
Observe that 
\[
 \zeta_{C_{p^{\nu}}}\m{\W}_{C_{p^kn}}(\m{R})(C_{p^{k-\nu}n}/C_{p^{k-\nu}}) \cong  \m{\W}_{C_{p^kn}}(\m{R})(C_{p^{k}n}/C_{p^{k}}).
\]

   So, the restriction maps of Definition \ref{defn:eqvt_Witt_restriction_maps} induce ring homomorphisms 
     \begin{equation*}
       r\colon \m{\W}_{C_{p^kn}}(\m{R})(C_{p^kn}/C_{p^k}) \to \m{\W}_{C_{p^{k-\nu}n}}(\m{R})(C_{p^{k-\nu}n}/C_{p^{k-\nu}}).
    \end{equation*}
\end{rmk}

\begin{rmk}
	We observe that the restriction on equivariant Witt vectors was defined \emph{as a morphism of $ C_{p^{k-\nu}n} $-Green functors}. Hence the restriction map $ r $ commutes with both $ F $ and $ V $, as they are part of the structure of the Green functors.  
\end{rmk}

We note that there is a clash of terminology here, as the restriction maps in the Mackey functor induce the Frobenius maps on equivariant Witt vectors rather than the maps that are called the restriction on equivariant Witt vectors. The use of these terms in this paper is consistent with the literature on topological Hochschild homology and Witt vectors. However, in order to avoid possible confusion, we point out that there are multiple uses of the term restriction that do not correspond to one another.

Finally, we define an equivariant analogue of the multiplicative lift $[-]_k\colon A \to W_k(A). $ 
\begin{defn}\label{defn:multiplicative_lift}
Let $\m{R}$ be a $C_n$-Tambara functor. For $m$ dividing $n$, we define the \emph{multiplicative lift map}
\[
 [-]_k:  \underline{R}(C_n / C_m) \to \underline{\W}_{C_{p^k n}} (\underline{R}) (C_{p^k n} / C_{p^km})
\]
as the composite
\begin{equation*}
\begin{tikzcd}[row sep=small]
	\m{R}(C_n/C_m)
	\ar[r,"\eta"] & i^*_{C_n}N_{C_n}^{C_{p^kn}}\m{R}(C_n/C_m) \ar[d,"{\rotatebox[origin=c]{270}{$\cong$}}" {xshift=2pt}] & & \\	 
     & N_{C_n}^{C_{p^kn}}\m{R}(C_{p^kn}/C_m) \ar[r,"n_{C_m}^{C_{p^km}}"] & N_{C_n}^{C_{p^kn}}\m{R}(C_{p^kn}/C_{p^km}) \ar[r,"q"] &
	  \m{\W}_{C_{p^kn}}(\m{R})(C_{p^kn}/C_{p^km}).
\end{tikzcd}
\end{equation*}

Here the first map $\eta$ is the unit map of the adjunction of Proposition \ref{prop:Tambara_norm_restriction_adjunction} and the isomorphism holds by definition of the restriction functor $i^*_{C_n}$. 
The map denoted $n_{C_m}^{C_{p^k}}$ is the internal norm from $C_m$ to $C_{p^km}$ in the Tambara functor $N_{C_n}^{C_{p^kn}}(\m{R})$, and the final map is the quotient map of \Cref{rmk:Witt_vectors_as_norm_quotient}.
Observe that since the internal norm map $n$ is multiplicative, but not additive, the multiplicative lift map fails to be additive.
\end{defn}

\begin{rmk}
    In \cite[Theorem 6.13]{BlGeHiLa}, for $\m{R}$ a $C_n$-Tambara functor, and $m$ dividing $ n$, the authors define a multiplicative lift map:
    \begin{align*}
      \underline{R}(C_n / C_m) \to
      \underline{\W}_{C_n} (\underline{R}) (C_n / C_n), 
    \end{align*}
    which they refer to as a ``Teichm\"uller map'' on equivariant Witt vectors. 
    The multiplicative lift maps needed for the current work are not of this form. For one thing, our multiplicative lift maps almost never land at orbits of the form $C_{n} / C_{n}$.  The only case where both constructions are defined is when $k=0$ and $m=n$, where the two constructions agree.
\end{rmk}

\begin{rmk}\label{rmk:externalizingnorms}
In \Cref{defn:multiplicative_lift}, we defined the multiplicative lift as a composite of maps between Tambara functors and an internal norm map in a given Tambara functor. Because Tambara functors are the $G$-commutative monoids in the $G$-symmetric monoidal category of $G$-Mackey functors \cite{Hoy}, this internal norm can be ``externalized'' using representability. See, for example, the discussion in \cite[Section 8]{ChanBiIncomplete}. An element $a\in N^{C_{p^kn}}_{C_n}\m{R}(C_{p^kn}/C_m)$ is represented by a map of $C_m$-Mackey functors
	\[\m{A}_{C_m}\xto{a} i_{C_m}^*N^{C_{p^kn}}_{C_n}\m{R}\]
	where $\m{A}_{C_m}$ is the Burnside Tambara functor for the cyclic group of order $m$.  The element $n_{C_m}^{C_{p^km}}(a)\in N^{C_{p^kn}}_{C_n}\m{R}(C_{p^kn}/C_{p^km})$ can then be defined as the element represented by the map of $C_{p^km}$-Mackey functors
	\[\m{A}_{C_{p^km}}\cong N^{C_{p^km}}_{C_m} \m{A}_{C_m}  \xto{N_{C_m}^{C_{p^km}}a} N_{C_m}^{C_{p^km}}i^*_{C_m}N^{C_{p^kn}}_{C_n}\m{R} \xto{\varepsilon} i^*_{C_{p^km}} N^{C_{p^kn}}_{C_n}\m{R},\]
	where $\m{A}_{C_{p^km}}$ is the $C_{p^km}$-Burnside Tambara functor and $\varepsilon$ is the counit of the $(N^{C_{p^km}}_{C_m},i^*_{C_m})$-adjunction between $C_m$- and $C_{p^km}$-Tambara functors.   This perspective on norms is useful in calculating multiplicative lifts in the proof of \Cref{prop:liftsagree} and later in \Cref{sec:TwistedTHHEquivariantWitt}.
\end{rmk}

Below we give two concrete examples of multiplicative lift maps of the form
\[
\m{R}(C_n/e)\to \m{\mathbb{W}}_{C_{p^kn}}(\m{R})(C_{p^kn}/C_{p^k}).
\] 
\begin{exmp}[Constant Tambara Functor]\label{const_mult_lift}
Let $p=3$, and consider the constant $C_2$-Tambara functor $\m{\mathbb{F}}_3$. The multiplicative lift map when $k=m=1$ is the map
\begin{center}
\begin{tikzcd}
    \m{\mathbb{F}}_3(C_2/e) \ar[r,"{[-]_1}"] & 
    \m{\mathbb{W}}_{C_{6}}(\m{\mathbb{F}}_3)(C_{6}/C_{3})
\end{tikzcd}
\end{center}
defined as the composite
\begin{equation*}
\begin{tikzcd}[column sep=small]
    \mathbb{F}_3= \m{\mathbb{F}}_3(C_2/e)
            \ar[r,"\eta"] & i^*_{C_2}N_{C_2}^{C_6}\m{\mathbb{F}}_3(C_2/e) \cong N_{C_2}^{C_6}\m{\mathbb{F}}_3(C_6/e) \ar[r,"n_e^{C_3}"] & N_{C_2}^{C_6}\m{\mathbb{F}}_3(C_6/C_3) \ar[r,"q"] & \m{\mathbb{W}}_{C_{6}}(\m{\mathbb{F}}_3)(C_{6}/C_{3}) \,.
\end{tikzcd} 
\end{equation*}
Using Proposition \ref{prop:norm_as_indexed_product},
we recognize the codomain of $\eta$ to be
    \begin{equation}
        i^*_{C_2}N_{C_2}^{C_6}\m{\mathbb{F}}_3\cong \m{\mathbb{F}}_3^{\square 3}.
    \end{equation}
    Evaluated on the orbit $C_2/e$, the box product reduces to the tensor product and $\eta\colon \mathbb{F}_3 \to \mathbb{F}_3^{\otimes 3}$ is the map $\alpha\mapsto (\alpha, 1, 1)$.
    The isomorphism $i^*_{C_2}N_{C_2}^{C_6}\m{\mathbb{F}}_3(C_2/e) \cong N_{C_2}^{C_6}\m{\mathbb{F}}_3(C_6/e)$ is the identity on $\mathbb{F}_3$. 

    To understand the internal norm map $n_e^{C_3}$, we can work in the Tambara functor $i_{C_3}^*N_{C_2}^{C_6}\m{\mathbb{F}}_3$. In particular, by \Cref{prop:norm_as_indexed_product},
    we obtain:
    \begin{align*}
        N_{C_2}^{C_6}(\m{\mathbb{F}}_3)(C_6/C_3) &= i^*_{C_3} N_{C_2}^{C_6}\m{\mathbb{F}}_3(C_3/C_3) \\
        &\overset{}{\cong} N_e^{C_3}(i^*_e \m{\mathbb{F}}_3)(C_3/C_3). 
        \end{align*}
     Proposition 5.23 of \cite{BlGeHiLa} then gives the identification $N_{C_2}^{C_6}\m{\mathbb{F}}_3(C_6/C_3) \cong \mathbb{Z}/9$, which can also be computed using the formulas of \cite[p.~90]{MThesis}. A similar string of identifications allows us to deduce that $n_e^{C_3}\colon N_{C_2}^{C_6} \m{\mathbb{F}}_3 (C_6 / C_e) \to N_{C_2}^{C_6} \m{\mathbb{F}}_3 (C_6 / C_3)$ is the map $n_e^{C_3}\colon {\mathbb{F}}_3 \to \mathbb{Z} / 9$ defined by $0 \mapsto 0,$ $1 \mapsto 1,$ and $-1 \mapsto -1.$

    As in Example \ref{trivial_weyl_Witt}, since the Weyl action is trivial, we have an isomorphism $\m{\HH}_{C_2}^{C_{6}}(\m{\mathbb{F}}_3)_0 \cong N_{C_2}^{C_6}\m{\mathbb{F}}_3$ and the quotient $N_{C_2}^{C_6}\m{\mathbb{F}}_3(C_6/C_3)\xto{q} {\m{\HH}_{C_2}^{C_6}}(\m{\mathbb{F}}_3)_0(C_6/C_3)$ is the identity on $\Z/9$.

    In summary, the multiplicative lift map is
\[              \mathbb{F}_3 \xto{\alpha\mapsto (\alpha\otimes 1\otimes 1)}  \mathbb{F}_3^{\otimes 3} \cong \mathbb{F}_3 \xto{n_e^{C_3}}  \Z/9 \xto{\id} \Z/9.
\]
The only nontrivial map here is $n_e^{C_3}$, and so the multiplicative lift is simply the ``obvious'' multiplicative map given by $0\mapsto 0$, $1\mapsto 1$, and $-1\mapsto -1$.  
\end{exmp}

\begin{exmp}[Burnside Tambara Functor]\label{burnside_mult_lift}
    The $k$th multiplicative lift map for $\m{A}_{C_n}$ is the map 
    \begin{center}
        \begin{tikzcd}
            \m{A}_{C_n}(C_n/e) \ar[r,"{[-]_k}"] & \m{\mathbb{W}}_{C_{p^kn}}(  \m{A}_{C_n})(C_{p^kn}/C_{p^k}) =\m{A}_{C_{p^kn}}(C_{p^kn}/C_{p^k})
        \end{tikzcd}
    \end{center}
given by the composite in \Cref{defn:multiplicative_lift}.
Since the Burnside Tambara functor is the unit for the box product and the norm preserves the unit,  $N_{C_n}^{C_{p^kn}}\m{A}_{C_n}\cong\m{A}_{C_{p^kn}}$.  Using \Cref{prop:norm_as_indexed_product},
we identify
\begin{equation*}
    i^*_{C_n} \m{A}_{C_{p^kn}} \cong \m{A}_{C_{n}}^{\square p^k} \cong \m{A}_{C_{n}}.
\end{equation*}
Under this identification, the map $\eta\colon \m{A}_{C_n}(C_n/e)\to i_{C_n}^*\m{A}_{C_{p^kn}}(C_{n}/e)$ is the identity.  The Weyl action on the Burnside ring Mackey functor is the trivial action and thus the quotient map $q\colon \m{A}_{C_{p^kn}}\to \m{\mathbb{W}}_{C_{p^kn}}(\m{A}_{C_n})$ is the identity as well.  Hence, the multiplicative lift $[-]_k\colon \m{A}_{C_n}(C_n/e)\to \m{\mathbb{W}}_{C_{p^kn}}(\m{A}_{C_n})(C_{p^kn}/C_{p^k})$ is simply the norm map
\[
 n_e^{C_{p^k}}:  \m{A}_{C_{p^kn}}(C_{p^kn}/e) \to  \m{A}_{C_{p^kn}}(C_{p^kn}/C_{p^k})
\]
By \cite[Example 1.4.6]{MThesis} , this internal norm map is given by
\begin{equation*}
\begin{split}
    A(e) &\overset{n_e^{C_{p^k}}}{\longrightarrow} A(C_{p^k}) \\
    [M] &\mapsto [\{\textnormal{functions of sets }C_{p^k} \to M\}],
\end{split}
\end{equation*}
where $ M $ denotes a finite set. Here $A(G)$ denotes the Burnside ring of the group $G$, as in Example \ref{ex:Burnside}.
\end{exmp}

Recall from Section \ref{introWitt} that classically the multiplicative lift map $[-]_k$ for Witt vectors lifts the $p^k$-power map. We now verify that our multiplicative lift on equivariant Witt vectors $[-]_k$ also lifts the $p^k$ power map. 
First, we recall that given an inclusion of a subgroup $ H \subset G $, the restriction functor $i_H^*\colon \Mack_G \to \Mack_H $ of Definition \ref{defn:restrictionMackey} admits both left and right adjoints, referred to as induction and coinduction respectively. 
Moreover, when $ G/H $ is finite, the induction and coinduction functors are canonically isomorphic. 
The unit of this adjunction, which we denote $\m{\res}$ following \cite[Definition 6.9]{BlGeHiLa}, is a kind of extension to Mackey functors of the internal Mackey functor restriction maps $\res_H^G.$ Suppose $\m{R}$ is a $C_n$-Mackey functor. On the $p$-typical equivariant Witt vectors of $\m{R}$, the unit of this adjunction gives a map 
\begin{equation}\label{eq:equivariant_Witt_adjunction_unit}
    \m{\res}\colon \m{\mathbb{W}}_{C_{p^kn}} (\m{R}) \to \operatorname{CoInd}_H^{C_{p^kn}} i^*_H \m{\mathbb{W}}_{C_{p^kn}} (\m{R}) 
\end{equation}
for any $ H \subset C_{p^kn}$.

In \cite[Proposition 5.19]{BlGeHiLa}, the authors show that the restriction $i_H^*\m{HC}_{C_n}^{C_{p^kn}}(\m{R})_{\bullet}$ of the simplicial Mackey functor from which $\m{\mathbb{W}}_{C_{p^kn}}(\m{R})$ is defined is a subdivision of the simplicial Mackey functor $\m{HC}_{C_n}^{H}(\m{R})_{\bullet} $ from which $\m{\mathbb{W}}_{H}(\m{R})$ is defined. Since the homology of the subdivided complex is isomorphic to the homology of the original complex, this shows that $i_{H}^*\m{\mathbb{W}}_{C_{p^kn}}(\m{R})\cong \m{\mathbb{W}}_{H}(\m{R})$.  Thus we may identify the target of $\m{\res}$ as $\CoInd_H^{C_{p^kn}}\m{\mathbb{W}}_H(\m{R}).$

\begin{prop} \label{prop:multLift}
Let $\m{R}$ be a $C_n$-Tambara functor and let $p$ be coprime to $n$. For any $m$ dividing $n$, the multiplicative lift map $[-]_k$ of Definition \ref{defn:multiplicative_lift} lifts the $p^k$th power map; i.e.~the following diagram commutes:
 \[
    \begin{tikzcd}
        & \m{\mathbb{W}}_{C_{p^kn}}(\m{R})(C_{p^kn}/C_{p^km}) \ar[d,"\m{\res}"] \\
        \m{R}(C_n/C_m) \ar[r,"r\mapsto r^{p^k}"] \ar[ur,"{[-]_k}"] & \m{\mathbb{W}}_{C_n}(\m{R})(C_n/C_m).
    \end{tikzcd}
\]
This diagram uses the natural isomorphism $\CoInd_{C_n}^{C_{p^kn}}\m{\mathbb{W}}_{C_n}(\m{R})(C_{p^kn}/C_{p^km})\cong
\m{\mathbb{W}}_{C_n}(\m{R})(C_n/C_m)$.
\end{prop}

\begin{proof}
We begin by analyzing the restriction map 
\[\m{\res}\colon \m{\mathbb{W}}_{C_{p^kn}}(\m{R})(C_{p^kn}/C_{p^km})\to \CoInd_{C_n}^{C_{p^kn}}\m{\mathbb{W}}_{C_n}(\m{R})(C_{p^kn}/C_{p^km})\cong
\m{\mathbb{W}}_{C_n}(\m{R})(C_n/C_m).\]  
This map is the composite 
\[\m{\mathbb{W}}_{C_{p^kn}}(\m{R})(C_{p^kn}/C_{p^km})\xto{\res} \m{\mathbb{W}}_{C_{p^kn}}(\m{R})(C_{p^kn}/C_m) = i^*_{C_n}\m{\mathbb{W}}_{C_{p^kn}}(\m{R})(C_n/C_m)\xto{\cong} \m{\mathbb{W}}_{C_n}(\m{R})(C_n/C_m)\]
of the internal restriction map in the Mackey functor $\m{\mathbb{W}}_{C_{p^kn}}(\m{R})$ and the natural isomorphism of Witt vectors of \cite[Proposition 5.19]{BlGeHiLa}.  The diagram of interest thus factors as follows.  We have omitted the isomorphism $i^*_{C_n}N_{C_n}^{C_{p^kn}}\m{R}(C_n/C_m)\cong N_{C_n}^{C_{p^kn}}\m{R}(C_{p^kn}/C_m)$ in the target of the map $\eta$.
\[
 \begin{tikzcd}
&&N_{C_n}^{C_{p^kn}}\m{R}(C_{p^kn}/C_{p^km}) \ar[r,"q"]\ar[d,"\res"]& \m{\mathbb{W}}_{C_{p^kn}}(\m{R})(C_{p^kn}/C_{p^km})\ar[d,"\res"] \\
   & N_{C_n}^{C_{p^kn}}\m{R}(C_{p^kn}/C_m)\ar[ur,"n"]\ar[r]&N_{C_n}^{C_{p^kn}}\m{R}(C_{p^kn}/C_m)\ar[r,"q"] & \m{\mathbb{W}}_{C_{p^kn}}(\m{R})(C_{p^kn}/C_m)\ar[d,phantom,"{\rotatebox{270}{$\cong$}}"]\\  
 \m{R}(C_n/C_m) \ar[ru,"\eta"]\ar[rrr,"x\mapsto x^{p^k}"]&&&\m{\mathbb{W}}_{C_n}(\m{R})(C_n/C_m)
\end{tikzcd}
\]
Here the upper right square commutes because $q$ is a map of Green functors; we define the horizontal map in the upper left triangle to be the composite $\res\circ n$ so that the triangle also commutes.  In any Tambara functor, the composite of the internal norm $n_H^K$ and restriction $\res_K^H$ is the product over the action of the  Weyl group $W_H(K)$.  Thus our unlabeled map is the product over the action of the Weyl group $W_{C_{p^km}}(C_m)\cong C_{p^km}/C_m$.  This group has order $p^k$, and so it suffices to observe that the action of each $g\in C_{p^km}/C_m$ descends to the identity once we pass to the Witt vectors.

 Because $C_m$ is a subgroup of $C_n$, we can identify the the lower trapezoid in the previous diagram with the diagram
\begin{equation}\label{trapezoiddiagramformultlifts}
\begin{tikzcd}
& i_{C_n}^*N_{C_n}^{C_{p^kn}}\m{R}(C_n/C_m) \ar[r, "\prod_g g\cdot-"]& i_{C_n}^* N_{C_n}^{C_{p^kn}}\m{R}(C_n/C_m)\ar[r,"q"] & i_{C_n}^*\m{\mathbb{W}}_{C_{p^kn}}(\m{R})(C_n/C_m)\ar[d,phantom,"{\rotatebox{90}{$\cong$}}"] \\
\m{R}(C_n/C_m) \ar[ur,"\eta"]\ar[rr,"x\mapsto x^{p^k}"]&&\m{R}(C_n/C_m) \ar[r, "q"]& \m{\mathbb{W}}_{C_n}(\m{R})(C_n/C_m),
\end{tikzcd}
\end{equation}
where we have explicitly factored the map along the bottom into the power map and the quotient to the Witt vectors.
While the map labeled $\prod_g g\cdot -$ and the $p^k$th power map do not come from maps of $C_n$-Mackey functors, all the remaining maps in this diagram do, and each $g$ individually acts through a map of $C_n$-Mackey functors.

By definition of the Witt vectors for Green functors, the maps labelled $q$ identify elements in the same orbit under the Weyl action.  Hence it suffices to observe that $\eta(x)$ is sent to $x$ under the right vertical isomorphism. This follows from the description of this isomorphism in \cite{BlGeHiLa}. 
\end{proof}

We illustrate the action of the multiplicative lift via two examples. 
\begin{exmp}[Constant Tambara Functor]
Recall the multiplicative lift map $[-]_1$ for the constant $C_2$-Tambara functor $\m{\mathbb{F}}_3$ from \Cref{const_mult_lift}. Proposition \ref{prop:multLift} ensures a commutative diagram  
  \begin{center}
        \begin{tikzcd}
            & \m{\mathbb{W}}_{C_{6}}(\m{\mathbb{F}}_3)(C_{6}/C_{3}) \ar[d,"\m{\res}"] \\ \m{\mathbb{F}}_3(C_2/e) \ar[ur,"{[-]_1}"] \ar[r,"r\mapsto r^{3}"'] 
            & \m{\mathbb{W}}_{C_2}(\m{\mathbb{F}}_3)(C_2/e).
        \end{tikzcd}
    \end{center}
We calculated in \Cref{const_mult_lift} that $[-]_1$ is the map $\mathbb{F}_3\to \mathbb{Z}/9$ given by $0\mapsto 0$, $1\mapsto 1$ and $-1\mapsto -1$.  We now calculate the map $\m{\res}$ to illustrate the commutativity in this explicit example.

The map $\m{\res}\colon \m{\mathbb{W}}_{C_6}(\m{\mathbb{F}}_3) (C_6 / C_3) \to \m{\mathbb{W}}_{C_2} ( \m{\mathbb{F}}_3) (C_2 / e)$ is given by
    \[\m{\mathbb{W}}_{C_6}(\m{\mathbb{F}}_3) (C_6 / C_3) \xrightarrow{\res_e^{C_3}} \m{\mathbb{W}}_{C_6}(\m{\mathbb{F}}_3) (C_6 / e) \cong i^*_{C_2} \m{\mathbb{W}}_{C_6}(\m{\mathbb{F}}_3) (C_2 / e) \to \m{\mathbb{W}}_{C_2} (\m{\mathbb{F}}_3) (C_2 / e), \]
        which simplifies, via \Cref{trivial_weyl_Witt} and \Cref{const_mult_lift},
        to
    \begin{center}
    \begin{tikzcd}[row sep=tiny]
    N_{C_2}^{C_6}\m{\mathbb{F}}_3(C_6/C_3) \ar[d,phantom,"{\cong}" {rotate=-90}] \ar[r,"\res_e^{C_3}"] & N_{C_2}^{C_6}\m{\mathbb{F}}_3(C_6/e) \ar[d,phantom,"{\cong}" {rotate=-90}] \\
    \mathbb{Z} / 9 \ar[r] & \mathbb{Z}/3.
    \end{tikzcd}
    \end{center}
    This map relates only the  $C_3$ and $e$ levels of the Tambara functor, so we can take the restriction  $i_{C_3}^* N_{C_2}^{C_6} \m{\mathbb{F}}_3$ and determine $\res_e^{C_3}$ there. 
    By \Cref{prop:norm_as_indexed_product},
   \[i_{C_3} N_{C_2}^{C_6} \m{\mathbb{F}}_3 \cong N_e^{C_3} (i_{e} \m{\mathbb{F}}_3), \]
   	so $\res_e^{C_3}$ must be reduction modulo 3 by \cite[Definition 2.2.1 (3)]{MThesis}.
    
The diagram is thus
\begin{center}
\begin{tikzcd}
& \mathbb{Z}/9\ar[d,"\m{\res}\,=\,\text{mod} \,3"]\\
\mathbb{F}_3 \ar[r,"r\mapsto r^{3}"'] \ar[ur,"{[-]}_1"]
& \mathbb{F}_3
\end{tikzcd}
\end{center}
Since the map $r\mapsto r^3$ is the identity in $\mathbb{F}_3$, this diagram clearly commutes.
\end{exmp}
 
\begin{exmp}[Burnside Tambara Functor]
Let $ n = 1 $. We consider the diagram in \ref{prop:multLift} for the case where $\m{R}$ is the Burnside Tambara functor for the trivial group (i.e.~ $\mathbb{Z}$) and $p$ and $k$ are arbitrary.  It is also straightforward to show directly that the diagram in \ref{prop:multLift} commutes in this case.  

The identifications used in \Cref{burnside_mult_lift} also show that the map $\m{\res}$ is the map $\res\colon A(C_{p^k})\to A(e)$ between Burnside rings.  In \Cref{burnside_mult_lift} we calculated that the multiplicative lift $[-]_k$ is the norm map $A(e)\to A(C_{p^k})$.  Hence the composite
\begin{equation*}
    \m{\res}\circ[-]_k \colon A(e) \xto{n^{C_{p^k}}_e} A(C_{p^k}) \xto{\res^{C_{p^k}}_{e}} A(e).
\end{equation*}
is the product over the (trivial) action of the Weyl group $W_{C_{p^k}}(e) = C_{p^k}$.  Since the Weyl group has cardinality $p^k$, we see that this composite reproduces the $p^k$th-power map on $A(e)=\Z$.
\end{exmp}

We now show that our equivariant multiplicative lift map generalizes the classical multiplicative lift map of Definition \ref{defn:multlift} in the following sense.

\begin{prop}\label{prop:liftsagree}
Let $n=1$, so $R$ is a Tambara functor for the trivial group, or in other words, a commutative ring. Then the multiplicative lift of \Cref{defn:multiplicative_lift}
\[
[-]_k: R(e/e) = R \to \underline{\W}_{C_{p^k}} (R) (C_{p^k} / C_{p^k}) \cong W_{k+1}(R)
\]
is the usual multiplicative lift map for $p$-typical Witt vectors under the identification of Example \ref{ex:Wittvec_as_e_eqvt}.
\end{prop}
\begin{rmk}\label{liftnames2}The careful reader may note a grading convention shift:  our multiplicative lift $[-]_k$ corresponds to the classical multiplicative lift that would most often be denoted in the literature by $[-]_{k+1}$. This is why we chose a different naming convention for the classical multiplicative lifts, as discussed in Remark \ref{liftnames}. The notation in this paper is chosen so that the index aligns well with the group $C_{p^k}$. 
\end{rmk}
\begin{proof}
The classical multiplicative lift map $[-]_{k}$ (which in the classical context is typically denoted $[-]_{k+1}$) is characterized by the conditions that it lifts the $p^k$-power map, in the sense that $F^{k}([a]_k) = a^{p^{k}},$ and that $R^k[a]_k = a$. We verify these conditions here to show that we recover the classical multiplicative lift map. 

By Proposition \ref{prop:multLift}, for $R$ a Tambara functor for the trivial group, the following diagram commutes, where $[-]_k$ denotes the equivariant multiplicative lift:
\[
\begin{tikzcd}
        & \m{\mathbb{W}}_{C_{p^k}}(R)(C_{p^k}/C_{p^k}) \ar[d,"\m{\res}"] \\
        R \ar[r,"r\mapsto r^{p^k}"] \ar[ur,"{[-]_k}"] & \m{\mathbb{W}}_{e}(R)(e/e).
    \end{tikzcd}
\]

We expand this diagram to
\[
\begin{tikzcd}[column sep=small]
       & & \m{\mathbb{W}}_{C_{p^k}}(R)(C_{p^k}/C_{p^k}) \ar[d,"\res_e^{C_{p^k}}"] \ar[r,"\cong"]& \m{\pi}_0^{C_{p^k}}\THH(R)(C_{p^k}/C_{p^k}) \ar[r,"\cong"] \ar[d,"\res_e^{C_{p^k}}"] &  \pi_0\THH(R)^{C_{p^k}}  \ar[d,"F^k"] & W_{k+1}(R) \ar[d,"F^k"] \ar[l,swap,"\cong"] \\
        & & \m{\mathbb{W}}_{C_{p^k}}(R)(C_{p^k}/e) \ar[r, "\cong"] \ar[d,"\cong"]& \m{\pi}_0^{C_{p^k}}\THH(R)(C_{p^k}/e)\ar[r,"\cong"] \ar[d,"\cong"]& \pi_0\THH(R)  \ar[d,"\id"]& R \ar[d,"\id"] \ar[l,swap,"\cong"] \\
        R \ar[rr,"r\mapsto r^{p^k}"] \ar[uurr,"{[-]_k}", bend left=15] & & \m{\mathbb{W}}_{e}(R)(e/e) \ar[r, "\cong"]& \m{\pi}_0^{e}\THH(R)(e/e)\ar[r,"\cong"] & \pi_0\THH(R)& R \ar[l,swap,"\cong"]. 
   \end{tikzcd}
\]
The upper right-hand square commutes by work of Hesselholt and Madsen \cite[Theorem 3.3]{HeMa97}. 
The upper middle square commutes by the definition of the map $F$. 
The upper left square commutes because the horizontal maps come from maps of Mackey functors \cite{BlGeHiLa}. The lower squares commute by inspection. 
Thus, when $C_n=e$, the equivariant multiplicative lift $[-]_k$ lifts the $p^k$-power map in the classical sense. 

We now consider the maps
\begin{equation} \label{diag:multrecovers}
  R \xrightarrow{[-]_k}\m{\mathbb{W}}_{C_{p^k}}(R)(C_{p^k}/C_{p^k}) \xrightarrow{r^k} \m{\mathbb{W}}_e(R)(e/e)= R  
\end{equation}
where $r^k$ here is $k$th iterate of the ring map $r$ of  \ref{rmk:restrictionmapsWitt}.  Since $n=1$, it follows that $\nu=1$ and the map $r$ has the form $\m{\mathbb{W}}_{C_{p^k}}(R)(C_{p^k}/C_{p^k})\to \m{\mathbb{W}}_{C_{p^{k-1}}}(R)(C_{p^{k-1}}/C_{p^{k-1}})$.
We show that the composite of (\ref{diag:multrecovers}) is equal to the identity on $ R $.  We first consider the case where $k = 1.$ Let $a \in R.$ Then $a$ is represented by a map $\m{A}_e = \mathbb{Z} \xrightarrow{a} R,$ and using the externalized norms perspective of Remark \ref{rmk:externalizingnorms} 
and the Yoneda lemma, $[a]_1 \in \m{\mathbb{W}}_{C_p} (R) (C_p / C_p)$ is
obtained by evaluating the composite
\[
\m{A}_{C_p} = N_e^{C_p} \m{A}_e \xrightarrow{Na} N_e^{C_p}  R \xrightarrow{N \eta} N_e^{C_p} i_e N_e^{C_p} R \xrightarrow{\varepsilon} N_e^{C_p} R \xrightarrow{q} \m{\mathbb{W}}_{C_p}(R)
\]
at $ [C_p/C_p] \in \m{A}_{C_p}(C_p/C_p) $ (see Example \ref{ex:Burnside}). 
Here $\eta$ and $\varepsilon$ are the unit and counit of the $(N_e^{C_p}, i_e)$-adjunction between $e$ and $C_p$-Tambara functors, so $ N_e^{C_p} i_e R \xrightarrow{N \eta} N_e^{C_p} i_e N_e^{C_p} R \xrightarrow{\varepsilon} N_e^{C_p} R$ is the identity, where $N$ is the norm $N_e^{C_p}$. Hence the composite representing $[a]_1 \in \m{\mathbb{W}}_{C_p} (R) (C_p / C_p)$ can be shortened to 
\[
\m{A}_{C_p} = N_e^{C_p} \m{A}_e \xrightarrow{Na} N_e^{C_p} R \xrightarrow{q} \m{\mathbb{W}}_{C_p}(R).
\]
By the definition of $r$ in \Cref{rmk:restrictionmapsWitt}, the element $ r([a]_1) \in \m{\mathbb{W}}_{e}(R)(e/e)\cong\Phi^{C_p} \m{\mathbb{W}}_{C_p} (R) (e/e)$ is the image of $ [C_p/C_p] \in \big(\zeta_{C_p} \m{A}_{C_p}\big)(e/e) \cong \m{A}_{C_p}(C_p/C_p) $ under the composite from the upper left to the lower center
\[\begin{tikzcd}
	& \zeta_{C_p} \m{A}_{C_p} \ar[r, "r"] \ar[d, "\zeta_{C_p} Na", swap] & \Phi^{C_p} \m{A}_{C_p} \ar[d, "\Phi^{C_p} Na", swap] \ar[r,equals] & \m{A}_e \ar[d,"a"] \\
	& \zeta_{C_p} N_e^{C_p}R \ar[r, "r"] & \Phi^{C_p} N_e^{C_p} R \ar[r,equals] & R \,,
 \end{tikzcd} 
\]
where the diagram is one of $ C_p/C_p $-Tambara functors, or ordinary commutative rings. 
Unpacking the definition of geometric fixed points, we see that the map $r$ on Burnside Mackey functors quotients out the image of the transfer.  Thus in this case, $ r \colon \m{A}_{C_p}(C_p/C_p) \cong \Z\{C_p/e\} \oplus \Z\{C_p/C_p\} \to \m{A}_e(e) \cong \Z $ satisfies $ r\left([C_p/C_p]\right) = 1 $. 
The commutativity of the right hand square follows from taking the normal subgroup $ N = G $ in Proposition \ref{propitem:norm_geom_fixpt}. 
Thus $ r([a]_1) = a $.   The case for general $k$ follows in a similar fashion.

The maps $r$ on equivariant Witt vectors of Remark \ref{rmk:restrictionmapsWitt} are compatible with the restriction maps $R$ on classical Witt vectors of Definition \ref{restrictionWitt} in the sense that the following diagram commutes. 
\[
\begin{tikzcd}[column sep=small]
      \m{\mathbb{W}}_{C_{p^k}}(R)(C_{p^k}/C_{p^k}) \ar[d,"r"] \ar[r,"\cong"]& \m{\pi}_0^{C_{p^k}}\THH(R)(C_{p^k}/C_{p^k}) \ar[r,"\cong"] \ar[d,"R"] &  \pi_0\THH(R)^{C_{p^k}}  \ar[d,"R"] & W_{k+1}(R) \ar[d,"R"] \ar[l,swap,"\cong"] \\
         \m{\mathbb{W}}_{C_{p^{k-1}}}(R)(C_{p^{k-1}}/C_{p^{k-1}}) \ar[r, "\cong"]& \m{\pi}_0^{C_{p^{k-1}}}\THH(R)(C_{p^{k-1}}/C_{p^{k-1}})\ar[r,"\cong"] & \pi_0\THH(R)^{C_{p^{k-1}}} & W_k(R) \ar[l,swap,"\cong"]. 
   \end{tikzcd}
\] 
The commutativity of this diagram follows from \cite[Theorem 3.3]{HeMa97} and \cite[Section 5]{BlGeHiLa}. Hence, the proof that $r^k[-]_k = \id$ above implies that the multiplicative lift $[-]_k$ followed by $R^k$ is the identity. Hence, the multiplicative lift  
\[
[-]_k: R(e/e) = R \to \underline{\W}_{C_{p^k}} (R) (C_{p^k} / C_{p^k}) \cong W_{k+1}(R)
\]
is the usual multiplicative lift map for $p$-typical Witt vectors. 
\end{proof}

\section{Equivariant Witt complexes}\label{sec:EquivariantWittCplx}
In this section we define the notion of a $C_n$-equivariant Witt complex and show that when $n=1$ this recovers the classical Witt complex of \Cref{defn:Wittcplx_classical}.
\begin{defn}\label{defn:equivariant_Witt_complex}
Let $p$ be an odd prime,  and $C_n$ a cyclic group of order $n$, where $n$ is relatively prime to $p$. Let $\nu$ be the order of the $p$th power map on $C_n$.
Let $\m{R}$ be a $C_n$-Tambara functor such that $\m{R}(C_n/C_m)$ is a $\mathbb{Z}_{(p)}$-algebra for each $C_m\subset C_n$. 
A \emph{$C_n$-equivariant Witt complex} over $\m{R}$ is the data of:
\begin{enumerate}
    \item a collection $ \big\{\m{E}_*^{C_{p^{s}n}}\big\} $  of  graded $C_{p^sn}$-Green functors, where $s \geq0$, subject to the compatibility condition
    \begin{equation*}\label{eq:eqvt_Witt_cplx_compatibilites}
        i^*_{C_{p^kn}} \m{E}_*^{C_{p^sn}} \cong \m{E}_*^{C_{p^kn}} \text{ for } k<s, 
    \end{equation*} 
\item \label{defn_eqvt_Witt_cplx:differential} for each subgroup $C_q$ of $C_{p^sn}$, a differential $d$  on $\m{E}_*^{C_{p^sn}}(C_{p^sn}/C_q)$ making each $\m{E}_*^{C_{p^sn}}(C_{p^sn}/C_q)$ into a differential graded ring, preserved under the compatibility maps in (1),
\item \label{defn_eqvt_Witt_cplx:r}for each $s \geq \nu$, a map of $ C_{p^{s-\nu}n}$-Green functors
\[ r\colon \zeta_{C_{p^\nu}} \m{E}_*^{C_{p^sn}} \to \m{E}_*^{C_{p^{s-\nu}n}}, \]
    \item   for each $s \geq 0$, a map of $C_{p^sn}$-Green functors
    \begin{equation*}\label{eq:eqvt_Witt_cplx_unit}
        \lambda\colon \m{\W}_{C_{p^sn}}(\m{R}) \to \m{E}_0^{C_{p^sn}}.     
    \end{equation*}
  
\end{enumerate}
This data is required to satisfy the following conditions. Let $H$ and $L$ be subgroups of $C_{p^sn}$ with $H\subset L$ and with index $[L:H]$. Also let $F$ denote any restriction map of the form $\res^{C_{{p^s}m}}_{C_{p^{s-1}m}}$ in the Green functor $\m{E}_*^{C_{p^sn}}$.  
These relations must hold for all choices of $m$ dividing $n$.
\begin{enumerate}[label=(\roman*)]     \item \label{enumitem:eqvt_Witt_cplx_basic_relations} $ \lambda r = r \lambda,$ $dr=rd,$  
    \item \label{enumitem:eqvt_Witt_cplx_FdV_FV} $\res_H^L\,d\,\tr_H^L=d $ and $\res_H^L\,\tr_H^L=[L:H]$,
    \item \label{enumitem:eqvt_Witt_cplx_div_power_rule} $ i^*_{C_{p^{k-1}n}}Fd\lambda([x]_k) = \lambda([x]_{k-1})^{p-1}d\lambda([x]_{k-1})$ for $x\in \m{R}(C_n/C_m)$.
\end{enumerate}
Here $[-]_k$ denotes the multiplicative lift of Definition \ref{defn:multiplicative_lift}.
\end{defn}

\begin{rmk}\label{rmapsgiveringmaps}
Note that the maps of $ C_{p^{s-\nu}n} $ Green functors \[ r\colon \zeta_{C_{p^\nu}} \m{E}_*^{C_{p^sn}} \to \m{E}_*^{C_{p^{s-\nu}n}} \] in the definition of an equivariant Witt complex yield maps of differential graded rings
\[
r \colon \m{E}_*^{C_{p^sn}}(C_{p^sn}/C_{p^s}) \to \m{E}_*^{C_{p^{s-\nu}n}}\left(C_{p^{s-\nu}n}/C_{p^{s-\nu}}\right)\ \]
after evaluation at the orbit $(C_{p^{s-\nu}n}/C_{p^{s-\nu}})$. These maps of differential graded rings are the equivariant analogue of the structure maps in the pro-differential graded ring $E_{\bullet}^*$ of a classical Witt complex (see Definition \ref{defn:Wittcplx_classical}). 
\end{rmk}

\begin{rmk}
In the case where $H$ is an index $p$ subgroup of $L$, if we denote the transfer $\tr_H^L$ by $V$ and continue to denote the restriction $\res_H^L$ by $F$, condition \ref{enumitem:eqvt_Witt_cplx_FdV_FV} of \Cref{defn:equivariant_Witt_complex} takes the form
\[ FdV=d \text{\quad and \quad} FV=p,\]
mirroring the conditions in the classical definition of a Witt complex.
\end{rmk}
\begin{rmk}
	 It is worth observing that the conditions $FdV=d$ and $FV=p,$ together with Frobenius reciprocity, imply that $dF=pFd$ and thus the differentials cannot make $\m{E}_*^{C_{p^sn}}$ into a differential graded Green functor.
\end{rmk}

When $n=1$, we now show that Definition \ref{defn:equivariant_Witt_complex} recovers the classical notion of a Witt complex, as defined in Definition \ref{defn:Wittcplx_classical}. 
An $e$-Tambara functor is a commutative ring $R$. Suppose that $ R $ is a $\mathbb{Z}_{(p)}$-algebra. An $e$-equivariant Witt complex over $ R $ is then a collection   $ \big\{\m{E}_*^{C_{p^s}}\big\} $ of graded $C_{p^s}$-Green functors, $s\geq0$,  with differential structures on $\m{E}^{C_{p^s}}_*(C_{p^s}/C_{p^q})$,  subject to the compatibility conditions and relations of Definition \ref{defn:equivariant_Witt_complex}.

\begin{prop}
Let $e$ denote the trivial group, and let $p$ be an odd prime. Let $R$ be a commutative $\mathbb{Z}_{(p)}$-algebra and $ \big\{\m{E}_*^{C_{p^s}}\big\} $ an $e$-equivariant Witt complex over $R$. 
Endow the differential graded ring $ B^\ast_{s +1} := \m{E}_\ast^{C_{p^s}}(C_{p^s}/C_{p^s}) $ with the structure of a pro-system $ B_{\bullet}^* \to B_{\bullet -1}^*$ with structure maps the maps $r$ of Remark \ref{rmapsgiveringmaps}.
Then the pro-differential graded ring $\left\{B_{\bullet}^*\right\}$ forms a classical Witt complex over $R$.
\end{prop}
\begin{proof}
When we specialize to the case $C_n=e$, note that all subgroups have $p$-power index, so it suffices to consider restriction or transfer between subgroups of index $p$, which we again denote by $F$ or $V$.  Furthermore, the order of the $p$th power map on $e$ is $1$, so $\nu=1$ in part (\ref{defn_eqvt_Witt_cplx:r}) of \Cref{defn:equivariant_Witt_complex}.

By definition of an $e$-equivariant Witt complex, we have maps of Green functors 
\[
    \lambda\colon \m{\W}_{C_{p^s}}(R) \to \m{E}_0^{C_{p^s}}  \,.   
\]
We evaluate at the orbit $C_{p^s}/C_{p^s}$ to get a map
\[
    \lambda\colon \m{\W}_{C_{p^s}}(R)(C_{p^s}/C_{p^s})\cong W_{s+1}(R) \to B^0_{s+1} = \m{E}_0^{C_{p^s}}(C_{p^s}/C_{p^s}),     
\]
as desired. This is a strict map of pro-rings, since $\lambda r = r \lambda$ in the equivariant Witt complex. 

    The maps $F$ and $V$ in the definition of an equivariant Witt complex specialize to maps 
    \begin{equation*}
    \begin{split}
      F \colon B_s^\ast \to B_{s-1}^\ast \\
      V \colon B_{s}^{\ast} \to B_{s+1}^\ast
    \end{split}
    \end{equation*}
    and the relations in the equivariant Witt complex specialize to the relations in the non-equivariant case. For relation (iii), recall from Proposition \ref{prop:liftsagree} that the equivariant multiplicative lift maps $[-]_k$ specialize to the classical multiplicative lifts on $p$-typical Witt vectors when $n=1$. Thus $ B^\ast_\bullet $ forms a classical Witt complex over $ R $. \end{proof}

\section{Twisted topological Hochschild homology and equivariant Witt complexes}\label{sec:TwistedTHHEquivariantWitt}
In this section we prove that the equivariant homotopy of twisted topological Hochschild homology forms an equivariant Witt complex. This is an equivariant analogue of the theorem of Hesselholt and Madsen recalled as Theorem \ref{classicalWittcxtheorem} above. 

We begin by defining the necessary maps of $C_{p^{s-\nu}n}$-Green functors $r$ from \Cref{defn:equivariant_Witt_complex}.

\begin{defn}\label{defn:twisted_THH_restriction}
    Let $ \m{R} $ be a $ C_n $-Tambara functor. 
    Let $ p $ be a prime not dividing $ n $ and let $ \nu $ be the order of the $ p $th power map on $ C_n $. 
    For each $ s \geq \nu $, define a map $r$ between the homotopy Green functors of the twisted topological Hochschild homology of $ \m{R} $ as a composite: 
\[r\colon \zeta_{C_{p^\nu}}\m{\pi}_*^{C_{p^{s}n}}(\THH_{C_n}\m{R}) \cong \m{\pi}_*^{C_{p^{s-\nu}n}}\big(\rho_{p^\nu}^*(\THH_{C_n}\m{R})^{C_{p^\nu}}\big) \to \m{\pi}_*^{C_{p^{s-\nu}n}}\big(\rho_{p^\nu}^*\Phi^{C_{p^\nu}}\THH_{C_n}\m{R}\big)\,.\] 
The left-hand equivalence exists by definition, and the latter map is obtained by taking the $ C_{p^{s-\nu}n}$-homotopy Green functor of the canonical morphism of $S^1$-ring spectra
\begin{equation}\label{canonicalmorph} 
\rho^*_{p^\nu}(\THH_{C_n}\m{R})^{C_{p^\nu}} \to \rho^*_{p^\nu}\Phi^{C_{p^\nu}} \THH_{C_n} \m{R} 
\end{equation}
 between categorical and geometric fixed points. 
\end{defn}

\begin{rmk}\label{designatedringmapsTHH} In the case that $n$ and $p$ are relatively prime, as discussed in \Cref{twistedpcyclotomicTHH}, $\THH_{C_n}\m{R}$ has a twisted $p$-cyclotomic structure.  This means there is an $S^1$-equivalence
    \[ \rho_{p^\nu}^* \Phi^{C_{p^\nu}} \THH_{C_n}\m{R} \simeq \THH_{C_n}\m{R}.\]
 In particular, the map $r$ above induces a map of $C_{p^{s-\nu}n}$-Green functors
    \[ r\colon  \zeta_{C_{p^\nu}}\m{\pi}_*^{C_{p^{s}n}}(\THH_{C_n}\m{R}) \to \m{\pi}_*^{C_{p^{s-\nu}n}}\big(\THH_{C_n}\m{R}\big).
\] 
 \end{rmk}

These maps $r$ are the designated maps $r$ we need in order to show that the equivariant homotopy of twisted topological Hochschild homology is an equivariant Witt complex.

\begin{thm}\label{thm:twisted_THH_is_eqvt_Witt}
    Let $ n $ be a positive integer and $ p $ an odd prime not dividing $ n $. 
  	For $\m{R}$ a $C_n$-Tambara functor such that $\m{R}(C_n/C_m)$ is a $\mathbb{Z}_{(p)}$-algebra for each $C_m\subset C_n$, the graded Green functors $ \{\underline{\pi}_*^{C_{p^sn}} \THH_{C_n} \m{R}\}_{s \geq 0} $ form a $C_n$-equivariant Witt complex over $\m{R}$.
   \end{thm}

\begin{proof} 
The homotopy Mackey functors $ \underline{\pi}_*^{C_{p^sn}} \THH_{C_n}(\m{R}) $ are graded $C_{p^sn}$-Green functors because $\THH_{C_n}(\m{R})$ is an equivariant ring spectrum.  As shown in \cite{AnBlGeHiLaMa}, the spectrum $\THH_{C_n}(\m{R})$ is an $S^1$-spectrum.  As explained in \cite[Section 2.1]{HeMa04}, an $S^1$-action on a spectrum on which $2$ acts invertibly induces a differential on homotopy groups.  This applies to each of the fixed point spectra $\THH_{C_n}(\m{R})^{G}$, for any finite subgroup $G\subset S^1$, because these spectra enjoy actions of $S^1/G\cong S^1$.  This provides the required differential structure.

The designated maps
\[ r\colon  \zeta_{C_{p^\nu}}\m{\pi}_*^{C_{p^{s}n}}(\THH_{C_n}\m{R}) \to \m{\pi}_*^{C_{p^{s-\nu}n}}\big(\THH_{C_n}\m{R}\big), \]
required in part (\ref{defn_eqvt_Witt_cplx:r}) of \Cref{defn:equivariant_Witt_complex} are defined using twisted $p$-cyclotomicity in Remark \ref{designatedringmapsTHH}.

To define the map
\[
     \lambda\colon \m{\mathbb{W}}_{C_{p^sn}} (\m{R}) \to \m{\pi}_0^{C_{p^sn}} \THH_{C_n} (\m{R})
     \]  
we recall that by \cite[Theorem 5.1]{BlGeHiLa} there is a natural isomorphism
\[
\m{\pi}_0^{C_{p^sn}} \THH_{C_n}(\m{R}) \cong \m{\HH}_{C_n}^{C_{p^sn}}(\m{R})_0. 
\]
Since by definition 
\[
\m{\mathbb{W}}_{C_{p^sn}} (\m{R}) = \m{\HH}_{C_n}^{C_{p^sn}}(\m{R})_0,
\]
this yields the desired map $\lambda$.

We now verify the relations in Definition \ref{defn:equivariant_Witt_complex}. 
The relation $\lambda r = r \lambda$ holds because both maps $r$ arise from the twisted $p$-cyclotomicity of Hochschild homology for Green functors.

The twisted $p$-cyclotomic structure on $\THH_{C_n}(\m{R})$ yields an $S^1$-equivariant restriction map
\[
R\colon \rho_{p^\nu}^*\THH_{C_n}(\m{R})^{C_{p^\nu}} \to \THH_{C_n}(\m{R}),
\] 
as in \Cref{rmk:restrictionfortwistedpcycloTHH}.  

We defined \[
r\colon \zeta_{C_{p^\nu}}\underline{\pi}_0^{C_{p^sn}} \THH_{C_n}(\m{R})\to \underline{\pi}_0^{C_{p^{s-\nu}n}} \THH_{C_n} (\m{R})
\]
to be the map induced by $ R $ on homotopy Mackey functors. As the maps $d$ arise from the $S^1$-action on $\rho_q^*(\THH_{C_n}(\m{R})^{C_q})$ for each $C_q\subset S^1$, this implies the relation $dr=rd$.

The proof that the relation $\res_H^L\,d\,\tr_H^L=d $ holds is the same as the proof that $FdV = d$ in the classical case, and can be found in \cite[Lemma 1.5.1]{Hesselholt1996}. 

We now verify the relation $\res_H^L\,\tr_H^L=[L:H]$. By the double coset formula for Mackey functors,
\[
\res_H^L\,\tr_H^L(x) \, = \, \sum_{\ell \in W_{L}(H)} \ell \cdot x.\]
Hence \[\res_H^L\,\tr_H^L(x) = x + \gamma \cdot x + \gamma^2 \cdot x + \ldots \gamma^{[L:H]-1} \cdot x,\]
where we write $ \gamma $ for a generator of the Weyl group $W_{L}(H) \cong L/H = \langle\gamma\rangle$. The fact that the $L$-action extends to the circle group $S^1$ implies that $\gamma$ acts trivially on homotopy groups. Therefore, the composite $\res_H^L\,\tr_H^L$ is multiplication by $[L:H]$.

The verification of the relation $i^*_{C_{p^{k-1}n}}F d \lambda ( [-]_k) = \lambda ([-]_{k -1})^{p-1} d \lambda ([-]_{k - 1})$ requires more explanation. We prove that this relation holds in \Cref{prop:twisted_THH_power_rule}. 
\end{proof}

\begin{prop}\label{prop:twisted_THH_power_rule}
  	Suppose $\m{R}$ is a $C_n$-Tambara functor  and let $ p $ be an odd prime not dividing $ n $. Suppose further that $\m{R}(C_n/C_m)$ is a $\Z_{(p)}$-algebra for each $m$ dividing $n$.   Then for all $ m $ dividing $ n $, $k \geq 1$, and $ a \in \m{R}(C_n / C_m) $, the relation
  	\begin{equation}\label{eq:twisted_THH_power_rule}
     i^*_{C_{p^{k-1}n}} F d \lambda [a]_{k} = \lambda ([a]_{k - 1})^{p - 1} d \lambda([a]_{k - 1})
    \end{equation}
  	holds in $ \m{\pi}_1^{C_{p^{k-1}n}}(\THH_{C_{n}}\m{R})(C_{p^{k-1}n}/C_{p^{k-1}m}).$
    \end{prop}

This proposition is an equivariant analogue of \cite[Lemma~1.5.6]{Hesselholt1996}. Before proving Proposition \ref{prop:twisted_THH_power_rule}, we first prove some auxiliary lemmas. Our first lemma allows us to understand the multiplicative lifts $[-]_k$ via maps of spectra.

\begin{lem}\label{lem:spectrumlevelrepnmultlift}
	Let $a\in\m{R}(C_n/C_m),$ let $\eta_{{p^{k}n}}\colon  H \m{R}\to i^*_{C_{n}} N^{C_{p^kn}}_{C_{n}} H \m{R}$
	be the unit of the $(N^{C_{p^kn}}_{C_{n}}, i^*_{C_{n}})$-adjunction, and let $\varepsilon_{{p^km}}$ be the counit of the $(N^{C_{p^km}}_{C_{m}},i^*_{C_{m}})$-adjunction. Represent $a$ as a map $a\in [\sphere_{C_m},i_{C_m}^* H\m{R}]_{C_m}$.  Then $[a]_k$ is the image of postcomposing the $C_{p^km}$-equivariant homotopy class 

	\begin{multline}\label{externalizednormliftinspectra_general_m}
		\sphere_{C_{p^km}} \cong N^{C_{p^km}}_{C_m} \sphere_{C_m} \xto{Na} N_{C_m}^{C_{p^km}}i^*_{C_m}H\m{R} \xto{Ni^*(\eta_{{{p^kn}}})} N_{C_m}^{C_{p^km}}i^*_{C_m}N^{C_{p^kn}}_{C_n}H\m{R} \\ 
		\xto{\varepsilon_{{{p^km}}}} i^*_{C_{p^km}}N^{C_{p^kn}}_{C_n}H\m{R}
	\end{multline}
with $i^*_{C_{p^km}}(q)$ where $q$ is
	\[ q\colon N^{C_{p^kn}}_{C_n}H\m{R}\to H(N^{C_{p^kn}}_{C_n}\m{R}) \to H\m{\mathbb{W}}_{C_{p^kn}}(\m{R}).\]
\end{lem}

\begin{proof}
	Recall from \Cref{rmk:externalizingnorms} that the norm of $\eta_{{p^kn}}(a)$ from $C_m$ to $C_{p^k m}$ is represented by the map of $C_{p^km}$-Mackey functors
	\begin{equation}\label{eq:norm_externalized_recalled}
    \m{A}_{C_{p^km}}\cong N^{C_{p^km}}_{C_m} \m{A}_{C_m}  \xto{N(\eta_{{{p^kn}}}\circ a)} N_{C_m}^{C_{p^km}}i^*_{C_m}N^{C_{p^kn}}_{C_n}\m{R} \xto{\varepsilon_{{{p^km}}}} i^*_{C_{p^km}} N^{C_{p^kn}}_{C_n}\m{R}.
  \end{equation}
	Since the norm of a Mackey functor can be defined as $\m{\pi}_0$ of the norm of the corresponding Eilenberg--MacLane spectrum \cite[Theorem 1.3]{ullman2013tambara}, the composite (\ref{eq:norm_externalized_recalled}) is equivalent to $\m{\pi}_0^{C_{p^km}}$ of the map 
	\[ N^{C_{p^km}}_{C_m} H\m{A}_{C_m}  \xto{N(\eta_{{{p^kn}}}\circ a)} N_{C_m}^{C_{p^km}}i^*_{C_m}N^{C_{p^kn}}_{C_n}H\m{R} \xto{\varepsilon_{{{p^km}}}} i^*_{C_{p^km}} N^{C_{p^kn}}_{C_n}H\m{R}\]
	of $ C_{p^km} $-spectra; the fact that  $\sphere_{C_{p^km}}$ is $(-1)$-connected and $\sphere_{C_{p^km}} \to H\m{A}_{C_p^km}$ induces the identity on $\m{\pi}^{C_{p^km}}_0$ finishes the proof.
\end{proof}

We will also need a further lemma to use in identifying the multiplicative lifts. The interesting part of the multiplicative lifts is in the construction before passage to the Witt vectors---that is, the map we obtain before composing with the map $q$ in \Cref{lem:spectrumlevelrepnmultlift}.  By applying our next lemma in the case $Q=N^{C_{p^{k-1}n}}_{C_n}H\m{R}$, we get an explicit identification of this portion of the multiplicative lifts.  This will be valuable in proving \Cref{prop:twisted_THH_power_rule}.

\begin{lem} \label{lem:idcounitunitmultliftcomp} Let $Q$ be a commutative $C_{p^{k-1}n}$-ring spectrum where $p$ is relatively prime to $n$. Let 
\[\eta_{{p^{k}n}}\colon  Q\to i^*_{C_{p^{k-1}n}} N^{C_{p^kn}}_{C_{p^{k-1}n}}Q\]
be the unit of the $(N^{C_{p^kn}}_{C_{p^{k-1}n}}, i^*_{C_{p^{k-1}n}})$-adjunction  and let $\varepsilon_{{p^k}}$ be the counit of the $(N^{C_{p^k}}_{C_{p^{k-1}}},i^*_{C_{p^{k-1}}})$-adjunction.
Then the restriction of the composite 
 \[ N^{C_{p^k}}_{C_{p^{k-1}}}i^*_{C_{p^{k-1}}}Q\xto{ N_{C_{p^{k-1}}}^{C_{p^k}} i^*_{C_{p^{k-1}}} (\eta_{{p^{k}n}})} N^{C_{p^k}}_{C_{p^{k-1}}}i^*_{C_{p^{k-1}}}N^{C_{p^kn}}_{C_{p^{k-1}n}}Q \xto{\varepsilon_{{p^k}}}i^*_{C_{p^k}}N^{C_{p^kn}}_{C_{p^{k-1}n}}Q\] 
to $C_{p^{k-1}}$-spectra 
\[ Q^{\sma p} \xto{(g^{j_1}\sma \dotsb\sma g^{j_{p}})\circ \varphi} Q^{\sma p},\]
where $g$ is a generator of $C_{p^{k-1}n}$, $\varphi$ is a permutation of the factors, and the $ j_i $ are integers.
\end{lem}
These powers and the permutation are determinable from the relationship between $n$ and $p$, but we will not need explicit formulas for them.

\begin{rmk}\label{choiceofcosetrepsrmk}
The explicit formula here relies on an explicit choice of construction of the norm $N_{C_{p^{k-1}n}}^{C_{{p^k}n}}Q$ because in order to calculate $\varepsilon_{p^k}$, we need to understand the $C_{p^k}$-action on $N_{C_{p^{k-1}n}}^{C_{p^kn}}Q$.  As discussed in \cite{BohmannRiehl}, a construction of the norm $N_H^GQ$ relies on an explicit choice of coset representatives for $H$ in $G$ to define the action of $G$ on $N_H^GQ$.  In our case, $G=C_{p^kn}$ and $H=C_{p^{k-1}n}$.  If we let $\gamma$ denote a generator of $C_{p^{k}n}$, then our construction uses the coset representatives $1,\gamma, \gamma^2,\dots, \gamma^{p-1}$.  There are other natural choices of coset representatives as well.  We use this particular set of coset representatives because it is most compatible with the construction of the twisted cyclic bar construction, which we view as the norm from $H$ to $S^1$ below.
\end{rmk}

\begin{proof}
We can understand the composite in the statement explicitly. The unit map $\eta_{{p^{k}n}}$  is given by
\[ Q\cong Q \sma \sphere^{\sma p-1}\xto{\id\sma 1^{\sma p-1}} Q^{\sma p}\]
where $\sphere$ is the $C_{p^{k-1}n}$-sphere spectrum and $1\colon\sphere\to Q$ is the unit map of the $C_{p^{k-1}n}$-ring spectrum $Q$.

Viewing $N^{C_{p^k}}_{C_{p^{k-1}}}i^*_{C_{p^{k-1}}} N^{C_{p^kn}}_{C_{p^{k-1}n}} Q$ as $p$ smash copies of $N^{C_{p^kn}}_{C_{p^{k-1}n}}Q$, the map
 \[\varepsilon_{{p^k}}\colon N^{C_{p^k}}_{C_{p^{k-1}}}i^*_{C_{p^{k-1}}} N^{C_{p^kn}}_{C_{p^{k-1}n}} (Q)\to i_{C_{p^k}}^* N^{C_{p^kn}}_{C_{p^{k-1}n}}Q\] 
is the composite
\[ \begin{tikzcd}
(N^{C_{p^kn}}_{C_{p^{k-1}n}}Q)^{\sma p} \ar[dr, "\id\sma h\sma h^2\sma \dots\sma h^{p-1}", swap]\ar[r,"\varepsilon_{p^k}"]& i_{C_{p^k}}^* N^{C_{p^kn}}_{C_{p^{k-1}n}}Q\\
&(N_{C_{p^{k-1}n}}^{C_{p^{k}n}}Q)^{\sma p}\ar[u,"\text{mult}"]
\end{tikzcd}
\]
where $h$ is the generator of $C_{p^{k}}\subset C_{p^kn}$ and $h^i\colon N_{C_{p^{k-1}n}}^{C_{p^{k}n}} Q \to  N^{C_{p^kn}}_{C_{p^{k-1}n}} Q$ is action by $h^i$. Note that by the construction of the indexed smash product $(N_{C_{p^{k-1}n}}^{C_{p^{k}n}} Q)$ with the choice of coset representatives from \Cref{choiceofcosetrepsrmk}, $h$ acts as the $n$-fold iterate of the cyclic permutation of the $p$ smash factors in $N^{C_{p^kn}}_{C_{p^{k-1}n}}Q\simeq Q^{\sma p}$, plus the application of a generator $g$ of $C_{p^{k-1}n}$ every time a factor moves from the last to the first position.  
For each $i=0,\dots, p-1$, the action of $h^i$ is of the form
 \[(g^{i_1}\sma \dotsb \sma g^{i_{p}})\circ (\tau)^{ni}\]
where $\tau$ is the permutation that brings the last factor to the front. One can determine a formula for the exact powers $i_j$ of $g$ that occur, but this is unnecessary for our argument. The key is that because $p$ is relatively prime to $n$, as we run over the actions of $h^i$, the $\tau^{ni}$ part yields all possible cyclic permutations.

Thus $\varepsilon_{{p^k}}$ is the product over all cyclic reorderings of $N_{C_{p^{k-1}n}}^{C_{p^{k}n}} Q$, up to actions of powers of $g\in C_{p^{k-1}n}$ in each smash factor.

Hence, the composite $\varepsilon_{{p^k}}\circ Ni(\eta_{p^{k}n})$ is given by 
\[ Q^{\sma p} \xto{(g^{j_1}\sma \dotsb\sma g^{j_{p}})\circ\varphi} Q^{\sma p};\]
in other words, this is the action by some power of the generator $g$ of $C_{p^{k-1}n}$ in each smash factor, up to a permutation $\varphi$ of the factors.
\end{proof}

\begin{exmp}
    Let $p=3$, $k=1$, and $n=2$.  We calculate $\varepsilon_{3}\circ N_{e}^{C_3}i^*_e(\eta_{6})$ when $Q=H\m{R}$ explicitly to illustrate the previous lemma.  In this case, the map $\eta_{6}$   is given by
    \[
    H\m{R} \cong H\m{R}\wedge \mathbb{S}^{\wedge 2}\xrightarrow{\id\wedge 1\wedge 1} H\m{R}^{\wedge 3},
    \]
    while the map $\varepsilon_{3}\colon N_e^{C_3} i_e^* N_{C_2}^{C_6}H\m{R} \to i^*_{C_3} N_{C_2}^{C_6}H\m{R}$ is given by the composition
    \[
    (N_{C_2}^{C_6}H\m{R})^{\wedge 3} \xrightarrow{\id\wedge h\wedge h^2} (N_{C_2}^{C_6}H\m{R})^{\wedge 3}
    \xrightarrow{\text{mult}} i^*_{C_3}N_{C_2}^{C_6} H\m{R},
    \]
    where $h$ is a generator of $C_3\subset C_6$.
    In each factor of $N_{C_2}^{C_6} H\m{R}$, we have an action by $\id$, $h$, or $h^2$.  View $N_e^{C_3}i_e^*N_{C_2}^{C_6}H\m{R}$ as nine indexed smash copies of $H\m{R}$,
\[ (H\m{R})_1\sma\dots \sma(H\m{R})_9,\]
with $\id$ acting on the first three smash factors, $h$ on the second three and $h^2$ on the third three. The identity, of course, preserves $H\m{R}_1\sma H\m{R}_2\sma H\m{R}_3$.   Observe that we can take $h$ to be the square of a generator of the $C_6$ action on $N^{C_6}_{C_2}H\m{R}$.  Thus  $h$ acts by cyclically rotating twice, with an action of the generator $g$ of $C_2$ on each factor that moves from last to first spot.  Although spectra do not have underlying sets, because these maps arise from permutations of smash factors and actions of $g$ on individual smash factors, $h$ and $h^2$ are best specified by writing what the corresponding permutation and action maps do to a tuple of elements in a product of $C_2$-sets.  From this perspective, the map $h$ is the action/permutation combination determined by sending
\[ (x_4,x_5,x_6)\mapsto (gx_5, gx_6, x_4)\]
and the map $h^2$ is the action/permutation combination determined by sending
\[ (x_7,x_8,x_9)\mapsto (g^2x_9, gx_7, gx_8)=(x_9,gx_7,gx_8).\]
Thus the map $1\sma h\sma h^2\colon (H\m{R})^{\sma 9}\to (H\m{R})^{\sma 9}$ is the action/permutation combination specified by
\[(x_1,\dots, x_9)\mapsto (x_1,x_2,x_3,gx_5,gx_6,x_4,x_9,gx_7,gx_8).\]
The counit $\varepsilon_{3}$ applies $1\sma h\sma h^2$ and then views $(H\m{R})^{\sma 9}$ as three smash copies of the ring spectrum $H\m{R}\sma H\m{R}\sma H\m{R}$ and multiplies these together.

If we apply $\varepsilon_{3}$ to the image of the unit map $N_e^{C_3}i_e^*(\eta_{6})$, we see in each of the  three smash copies of $H\m{R}\sma H\m{R}\sma H\m{R}$, only one factor will not be the image of the unit $\sphere\to H\m{R}$.  The action of $g$ on the image of the unit is trivial and multiplication with the unit is also trivial.  This allows us to conclude that the overall map is
\[
\varepsilon_{3}\circ N^{C_3}_ei^*_e(\eta_{6})\colon H\m{R}^{\wedge 3} \xrightarrow{(\id\wedge g\wedge \id)\circ \varphi} H\m{R}^{\wedge 3}
\]
where $\varphi$ is the permutation $(1)(23)$.

\end{exmp}

Equipped with these lemmas, we proceed to the proof of \Cref{prop:twisted_THH_power_rule}.

\begin{proof}[{Proof of \Cref{prop:twisted_THH_power_rule}}]
In order to simplify the discussion, we first concentrate on the special case of the relation $i^*_{C_{p^{k-1}n}}F d \lambda [a]_k = \lambda ([a]_{k - 1})^{p-1} d (\lambda[a]_{k-1})$ when $k=1$ and $a\in \m{R}(C_n/e)$.  We will see that the cases $k \in \mathbb{Z}_{> 1}$ and $C_m\subset C_n$ are similar to this special case.  Our overall strategy is to realize both sides of the relation as maps of simplicial spectra and show that their geometric realizations are homotopic.

We begin by analyzing the left hand side of the relation, $i^*_{C_n}F d \lambda [a]_1.$ For $a\in \m{R}(C_n/e)$, we view $a$ as a map $\sphere\xto{a} i_e^*H\m{R}$ and construct $\lambda[a]_1$ as the homotopy class of a $C_{p}$-equivariant map \newline $\sphere_{C_{p}}\to i_{C_{p}}^*(\THH_{C_n}\m{R}).$ Given $a\in [\sphere, i_e^*H\m{R}]_e$, by \Cref{lem:spectrumlevelrepnmultlift} and \Cref{rmk:norm_to_twisted_HH_quotient} the element $\lambda[a]_1\in [\mathbb{S}_{C_{p}},i_{C_{p}}^*\THH_{C_n}\m{R}]_{C_{p}}$ is the $ C_p $-equivariant homotopy class of the composite map
\begin{multline*}
	\mathbb{S}_{C_{p}} \cong N^{C_{p}}_{e} \mathbb{S} \xto{N_e^{C_{p}}a} N_{e}^{C_{p}}i^*_{e}H\m{R}  \xto{N_e^{C_{p}}i_e^*(\eta_{{pn}})} N_{e}^{C_{p}}i^*_{e}N^{C_{pn}}_{C_n}H\m{R} \\
	\xto{\varepsilon_{p}} i^*_{C_{p}}N^{C_{pn}}_{C_n}H\m{R} 
	\xto{\eta_{S^1}} i^*_{C_{p}}N^{S^1}_{C_{pn}}N^{C_{pn}}_{C_n}H\m{R} \cong i^*_{C_{p}}N^{S^1}_{C_n}H\m{R} \cong i^*_{C_p} \THH_{C_n}\m{R},
\end{multline*}

where $\eta_{{pn}}$ is the unit of the $(N_{C_n}^{C_{pn}}, i^*_{C_n})$-adjunction, $\varepsilon_{p}$ is the counit of the $(N_e^{C_p}, i^*_e)$-adjunction, and $\eta_{S^1}$ is the unit of the $(N_{C_{pn}}^{S^1}, i^*_{C_{pn}})$-adjunction. 

We use the identification $\THH_{C_n}\m{R}\simeq N^{S^1}_{C_n}H\m{R}$ from \cite{AnBlGeHiLaMa}, together with the equivalence of left adjoints $N^{S^1}_{C_{pn}}N^{C_{pn}}_{C_n}\cong N^{S^1}_{C_n}$.

The map $\eta_{S^1}$ has a simplicial description. Regarding $ N^{C_{pn}}_{C_n} H\m{R}$ as a constant simplicial object, the map $\eta_{S^1}$ is the geometric realization of the map of simplicial $C_{pn}$-spectra
\[ N^{C_{pn}}_{C_n} H\m{R}\xto{\eta_{S^1}} N^{\cy, C_{pn}}_\bullet N^{C_{pn}}_{C_n}H\m{R}.\]
given on level $q$ by 
\[N^{C_{pn}}_{C_n} H\m{R}\cong N^{C_{pn}}_{C_n}H\m{R}\sma \sphere^{\sma q}\xto{\id\sma 1\sma\dots \sma 1} \left(N^{C_{pn}}_{C_n} H\m{R}\right)^{\sma q+1}.\]

We next turn to the application of the differential $d$. Recall that the differential comes from the $S^1$-action on $\THH_{C_n}\m{R}$.  Because we're working $C_p$-equivariantly, we give a simplicial description of this action that is explicitly $C_p$-equivariant.

The explicit description of this action requires a slightly different model for $N^{S^1}_{C_{pn}}N^{C_{pn}}_{C_n}H\m{R}\cong N^{S^1}_{C_n}H\m{R}$. As above,  $N^{S^1}_{C_{pn}}N^{C_{pn}}_{C_n}H\m{R}$ is the geometric realization of $N^{\cy,C_{pn}}_\bullet N_{C_n}^{C_{pn}}H\m{R}$.
There is an isomorphism 
\begin{equation}\label{Lambdapnopisomorphism}  i^*_{C_p}N^{\cy,C_{pn}}_\bullet N^{C_{pn}}_{C_n} H\m{R} \cong i^*_{C_p}\sd_{p}  N^{\cy,C_n}_\bullet H\m{R}
\end{equation}
of $\Lambda_{pn}^\op$-diagrams in $C_p$-spectra. In fact, this isomorphism is the simplicial manifestation of the isomorphism of left adjoints $N^{S^1}_{C_{pn}}N^{C_{pn}}_{C_n}\cong N^{S^1}_{C_n}$ mentioned above.  
On simplicial level $ q $, this is the standard shuffle:
\begin{equation*}
    i^*_{C_p} \left(N^{C_{pn}}_{C_n} H\m{R}\right)^{\sma (q+1)} = i^*_{C_p}\left(H\m{R}^{\sma p}\right)^{\sma (q+1)} \cong  i^*_{C_p}\left(H\m{R}^{\sma (q+1)}\right)^{\sma p} \,.
\end{equation*}

The $S^1$-action on $\lvert\sd_pN^{\cy,C_{n}}_\bullet H\m{R}\rvert$ can be described simplicially as a map
\begin{equation}\label{eq:simplicial_circle_action}
    \alpha_\bullet\colon S^1_{pn\bullet}\sma \sd_{p}N^{\cy,C_n}_\bullet H\m{R}\to \sd_{p}N^{\cy,C_n}_\bullet H\m{R}
\end{equation}
of $\Lambda_{pn}^\op$-spectra.  Here $S^1_{pn\bullet}$ is a $\Lambda_{pn}^\op$-model of the simplicial circle;  for any integer $l$, we let $S^1_{l\bullet}$ denote the $l$-fold subdivision of the standard simplicial model of the circle $S^1_\bullet$.  The cyclic structure maps on $S^1_{pn\bullet}\sma \sd_pN^{\cy,C_n}_\bullet H\m{R}$ are given as in \cite[\S1.4, p.~15]{Hesselholt1996}.
Explicitly, at simplicial level $q$, we view $C_{pn(q+1)}\sma (H\m{R}^{\sma p(q+1)}) \cong \bigvee_{C_{pn(q+1)}} H\m{R}^{\sma p(q+1)}$ and, as is implicit in \cite[\S 1.4]{Hesselholt1996}, the map
\[\bigvee_{C_{pn(q+1)}} H\m{R}^{\sma p(q+1)} \xto{\alpha} H\m{R}^{\sma p(q+1)}\]
is given on the wedge summand indexed by the generator of $C_{pn(q+1)}$ by the map that brings the first smash factor cyclically around to the last and acts by $g^{-1}$, where $g$ is the generator of $C_n$.  The map on the other wedge summands is determined by $C_{pn(q+1)}$-equivariance.
We may thus write the map representing $d\lambda[a]_1$ 
as the homotopy class in $[\Sigma^{\infty}_{C_p}S^1,i^*_{C_{p}}\THH_{C_n}\m{R}]_{C_{p}}$ of the geometric realization of the composite \begin{multline*}
 S^1_{pn\bullet}\sma N^{C_{p}}_e\sphere \xto{Na} S^1_{pn\bullet}\sma N^{C_{p}}_ei^*_eH\m{R}\xto{\varepsilon_{p}\circ \eta_{{pn}}} S^1_{pn\bullet}\sma i_{C_{p}}^*N^{C_{pn}}_{C_n}H\m{R}\\
\xto{\eta_{S^1}} S^1_{pn\bullet}\sma i^*_{C_{p}}N^{\cy,C_{pn}}_\bullet N^{C_{pn}}_{C_n}H\m{R}
\xto{\cong}S^1_{pn\bullet}\sma \sd_{p}N^{\cy,C_n}_\bullet H\m{R}\xto{\alpha} \sd_{p}N^{\cy, C_n}_\bullet H\m{R}
\end{multline*}
of maps of simplicial $C_{p}$-spectra. 
By convention, absence of a $\bullet$ indicates a constant simplicial object.

Since $F$ is the restriction map on graded Green functors, the operation $i^*_{C_n}F$ is given by the map 
\[ [\Sigma^{\infty}_{C_p}S^1,i_{C_p}^*\THH_{C_n}\m{R}]_{C_{p}}\to [\Sigma^{\infty}S^1,i_e^*\THH_{C_n}\m{R}]_e \]
that forgets from $C_{p}$-equivariant homotopy classes to underlying homotopy classes of maps. This means that $i^*_{C_n}Fd\lambda[a]_1$ is the non-equivariant homotopy class of the geometric realization of the composite
\begin{equation}\label{FDlambdax1compositemap}
\begin{split}
 S^1_{pn\bullet}\sma N^{C_{p}}_e\sphere \xto{Na} S^1_{pn\bullet}\sma N^{C_{p}}_ei^*_eH\m{R}\xto{\varepsilon_{p} \circ \eta_{{pn}}} S^1_{pn\bullet}\sma i_{C_{p}}^*N^{C_{pn}}_{C_n}H\m{R}\xto{\eta_{S^1}} S^1_{pn\bullet}\sma i^*_{C_{p}}N^{\cy,C_{pn}}_\bullet N^{C_{pn}}_{C_n}H\m{R}\\
\xto{\cong}S^1_{pn\bullet}\sma \sd_{p}N^{\cy,C_n}_\bullet H\m{R}\xto{\alpha} \sd_{p}N^{\cy, C_n}_\bullet H\m{R}\xto{d_0^{(\bullet+1)(p-1)}} N^{\cy,C_n}_\bullet H\m{R}
\end{split}
\end{equation}
where $d_0^{(\bullet+1)(p-1)}$ is the usual simplicial map that on realization is homotopic to the equivalence $\big| \sd_{p}N^{\cy,C_n}_\bullet H\m{R}\big|\to \big|N^{\cy,C_n}_\bullet H\m{R}\big|$. 
On simplicial level $q$, the map $d_0^{(q+1)(p-1)}$ multiplies the left-most smash factors of $H\m{R}^{\sma(q+1)p}$ until there are only $q+1$ factors remaining.

Using \Cref{lem:idcounitunitmultliftcomp} and unravelling the definition above we see that the composite
\[N^{C_p}_ei^*_{e}H\m{R}\xto{\eta_{S^1}\circ \varepsilon_{C_p} \circ\eta_{C_{pn}}} \sd_pN^{\cy,C_n}_\bullet H\m{R}\]
is given on simplicial level $q$ by the map
\[ \underbrace{H\m{R}\sma \dots \sma H\m{R}}_{\text{$p$ times}}\xto{(g^{j_1}\sma 1^{\sma q})\sma \dots \sma (g^{j_p}\sma 1^{\sma q})\circ\varphi} \underbrace{(H\m{R}^{\sma q+1})\sma \dots \sma (H\m{R}^{\sma q+1})}_{\text{$p$ times}} \]
where again $1\colon \sphere \to H\m{R}$ denotes the unit and $\varphi$ is some permutation of the $p$ smash factors.  At simplicial level $q$, the entire composite (\ref{FDlambdax1compositemap}) is thus given by the map
\begin{equation}\label{unpackingofFlambdaa1oncomponents}
\begin{tikzcd}
\underbrace{\sphere\sma \dots \sma\sphere}_{\text{$p$ times}} \ar[r, "{a\sma \dots \sma a}"] &\underbrace{H\m{R}\sma \dots \sma H\m{R}}_{\text{$p$ times}}\ar[r,"{(1^{\sma q}\sma g^{j_2})\sma \dots\sma (1^{\sma q}\sma g^{j_p}) \sma (1^{\sma q}\sma g^{j_1-1})\circ \varphi}"] & [+13em]\underbrace{(H\m{R}^{\sma q+1})\sma \dots \sma (H\m{R}^{\sma q+1})}_{\text{$p$ times}}\ar[d,swap,"{\text{mult}}"]\\
&& \underbrace{H\m{R}\sma \dots \sma H\m{R}}_{\text{$q+1$ times}}
\end{tikzcd}
\end{equation}
on the wedge summand of $S^1_{pn\bullet}\sma N^{C_p}_e\sphere$  indexed by the generator of $C_{pn(q+1)}$.  The map labelled $\text{mult}$ multiplies all the smash factors except for the last $q$.  The key observation is that this multiplication produces a product of $p-1$ images of $a$ under powers of $g$.

We next turn to the right hand side of the relation, $\lambda[a]_0^{p-1}d\lambda[a]_0$, again in the case where $a\in \m{R}(C_n/e)$.  Recalling that $[-]_0$ is the identity, a similar analysis shows that $\lambda[a]_0^{p-1}d\lambda[a]_0$ is the realization of the simplicial map
\begin{multline}\label{lambdap-1adlambdaa}
 S^1_{n\bullet}\sma \sphere^{\sma p}\xto{\id \sma a^{\sma p}} S^1_{n\bullet}\sma H\m{R}^{\sma p}\xto{\id \sma \eta_{S^1}^{\sma p}} S^1_{n\bullet}\sma (N^{\cy, C_n}_\bullet H\m{R})^{\sma p}\xto{\id\sma\text{mult}\sma \id} S^1_{n\bullet}\sma N^{\cy, C_n}_\bullet H\m{R}\sma N^{\cy, C_n}_\bullet H\m{R}\\
 \xto{\cong} N^{\cy, C_n}_\bullet H\m{R}\sma S^1_{n\bullet} \sma N^{\cy, C_n}_\bullet H\m{R} \xto{\id\sma \alpha} N^{\cy,C_n}_\bullet H\m{R}\sma N^{\cy,C_n}_\bullet H\m{R}\xto{\mu} N^{\cy, C_n}_\bullet H\m{R}.
 \end{multline}
 Here, the $\text{mult}$ map multiplies the first $p-1$ copies of $N^{\cy, C_n}_\bullet H\m{R}$. Looking just at the smash factors of $\sphere$ being multiplied, at simplicial level $q$  the composite $\text{mult} \circ \eta_{S^1}^{\sma p-1} \circ a^{\sma p-1} $ is equivalent to the map
\[ \sphere^{\sma q+1}\xto{a^{p-1}\sma 1^{\sma q}} \underbrace{H\m{R}\sma \dots \sma H\m{R}}_{\text{$q+1$ times}}.\]
As in the subdivided case of (\ref{eq:simplicial_circle_action}), the action map $\alpha_\bullet\colon S^1_{n\bullet}\sma N^{\cy, C_n}_\bullet H\m{R}\to N^{\cy, C_n}_\bullet H\m{R}$ is induced on simplicial level $q$ on the wedge summand of 
\[\bigvee_{C_{n(q+1)}} H\m{R}^{\sma q+1}\]
indexed by the generator of $C_{n(q+1)}$ by the map that sends the first smash factor to the last and acts via $g^{-1}$, where $g$ is the generator of $C_n$.
Ultimately, $\lambda[a]_0^{p-1}d\lambda[a]_0$ is the homotopy class of the realization of the map that on simplicial level $q$ and the wedge summand of $\bigvee_{C_{n(q+1)}}\sphere$ indexed by the generator is given by
\[\sphere^{\sma q+1}\xto{a^{p-1}\sma 1^{\sma q-1}\sma g^{-1}a} H\m{R}^{\sma q+1}.\]
While the maps (\ref{FDlambdax1compositemap}) and (\ref{lambdap-1adlambdaa}) do not agree on the nose, they do represent homotopy equivalent elements: once we geometrically realize and obtain the $S^1$-action, the action of an element of $C_n\subset S^1$ is homotopic to the identity. This verifies the identity when $k=1$.

We now analyze $i^*_{C_{p^{k-1}n}}F d \lambda [a]_k$ when $k > 1.$ For $a\in \m{R}(C_n/e)$, we view $a$ as a map $\mathbb{S}\xto{a} i_e^*H\m{R}$.  Then the element $\lambda[a]_k\in [\sphere_{C_{p^k}} ,i_{C_{p^k}}^*\THH_{C_n}\m{R}]_{C_{p^k}}$ is the class of the composite
\begin{multline} \label{eq:compLambdaakbig}
\sphere_{C_{p^k}}\cong N^{C_{p^k}}_{e} \mathbb{S} \xto{N_e^{C_{p^k}}a} N_{e}^{C_{p^k}}i^*_{e}H\m{R} \xto{N_e^{C_{p^k}}i^*_e(\eta_{p^kn})} N_{e}^{C_{p^k}}i^*_{e}N^{C_{p^kn}}_{C_n}H\m{R}\\\xto{\varepsilon_{p^k}} i^*_{C_{p^k}}N^{C_{p^kn}}_{C_n}H\m{R} 
\xto{\eta_{S^1}} i^*_{C_{p^k}}N^{S^1}_{C_{p^kn}}N^{C_{p^kn}}_{C_n} H\m{R}.
\end{multline}
Since we need to compare the left and right hand sides of our relation, we identity a relationship between $[a]_k$ and $[a]_{k-1}$ implicit in composite (\ref{eq:compLambdaakbig}).  Observe that the diagram 
\begin{equation}
\begin{tikzcd} \label{diag:identifyLifts}
	N^{C_{p^k}}_ei^*_eH\m{R}\ar[dd, "\cong"]\ar[r,"\eta_{p^kn}"]  & N^{C_{p^k}}_ei^*_eN^{C_{p^kn}}_{C_n} H\m{R}\ar[d,"\cong"]\ar[r,"\varepsilon_{{p^k}}"] &i_{C_{p^k}}^*N^{C_{p^kn}}_{C_n}H\m{R} \\
	& N^{C_{p^k}}_{C_{p^{k-1}}}N^{C_{p^{k-1}}}_e i^*_eN^{C_{p^kn}}_{C_n}H\m{R} \ar[r,"\varepsilon_{{p^{k-1}}}"]& N^{C_{p^k}}_{C_{p^{k-1}}}i_{C_{p^{k-1}}}^*N^{C_{p^kn}}_{C_{p^{k-1}n}}N^{C_{p^{k-1}n}}_{C_n}H\m{R}\ar[u,"\varepsilon_{{p^{k}}}"]\\
	N^{C_{p^k}}_{C_{p^{k-1}}}N^{C_{p^{k-1}}}_e i^*_eH\m{R}\ar[r,"\eta_{{p^{k-1}n}}"] & N^{C_{p^k}}_{C_{p^{k-1}}} N^{C_{p^{k-1}}}_ei^*_eN^{C_{p^{k-1}n}}_{C_n} H\m{R}\ar[r,"\varepsilon_{{p^{k-1}}}"] \ar[u,"\eta_{{p^kn}}"]& N^{C_{p^k}}_{C_{p^{k-1}}}i_{C_{p^{k-1}}}^*N^{C_{p^{k-1}n}}_{C_n}H\m{R}\ar[u, "\eta_{{p^kn}}"]
\end{tikzcd}
\end{equation}
commutes by the compatibility of the norm/restriction adjunctions between varying subgroups.  (Here for ease of notation we have simplified names of maps such as $N_{e}^{C_{p^k}}i^*_e\eta_{{p^kn}}$ to $\eta_{{p^kn}}$.)

The map across the bottom row is the norm of the ``$\varepsilon\circ \eta$'' that appears in the construction of $[a]_{k-1}$  in (\ref{externalizednormliftinspectra_general_m}),  while the one across the top row is that appearing in the construction of $[a]_k$. By Lemma \ref{lem:idcounitunitmultliftcomp} applied to the $C_{p^{k-1}n}$-spectrum $Q=N_{C_n}^{C_{p^{k-1}n}}H\m{R}$, the right vertical map is given by
\[(N^{C_{p^{k-1}n}}_{C_n}H\m{R})^{\sma p} \xto{(g^{j_1}\sma\dots\sma g^{j_p})\circ\varphi} (N^{C_{p^{k-1}n}}_{C_n}H\m{R})^{\sma p}
.\]

Composing the top row of Diagram \ref{diag:identifyLifts} with $\eta_{S^1}$ produces $\lambda[a]_k$.  When we include this additional map in Diagram (\ref{diag:identifyLifts}), we obtain the diagram
\[\begin{tikzcd}
 	N^{C_{p^k}}_e i^*_eH\m{R}\ar[d, "\cong"]\ar[r,"\varepsilon \circ \eta"] &i_{C_{p^k}}^*N^{C_{p^kn}}_{C_n}H\m{R} \ar[r, "\eta_{S^1}"] & i^*_{C_{p^k}} N_{C_{p^k n}}^{S^1} N_{C_n}^{C_{p^k n}} H \m{R} \\
	N^{C_{p^k}}_{C_{p^{k-1}}}N^{C_{p^{k-1}}}_e i^*_eH\m{R}\ar[r,"\varepsilon \circ \eta"] &  N^{C_{p^k}}_{C_{p^{k-1}}}i_{C_{p^{k-1}}}^*N^{C_{p^{k-1}n}}_{C_n}H\m{R}\ar[u, "\varepsilon \circ \eta"] &
\end{tikzcd}
\]
which identifies $\lambda [a]_k$ as a twisted normed power of $\lambda [a]_{k-1}.$

From here, the analyzing the differential $d$ applied to $\lambda [a]_k$ proceeds essentially as in the $k=1$ case; note that the action $\alpha$ involves multiplication by $g^{-1}$ where $g$ is a generator of $C_{p^{k-1}n}$ rather than $C_{n}$. Since  $F$ is the restriction map on graded Green functors, $i^*_{C_{p^{k-1}n}}F$ is given by the map 
\[ [\Sigma^\infty_{C_{p^k}} S^1,i_{C_{p^k}}^*\THH_{C_n}\m{R}]_{C_{p^k}}\to [\Sigma^\infty_{C_{p^{k-1}}} S^1,i_{C_{p^{k - 1}}}^*\THH_{C_n}\m{R}]_{C_{p^{k - 1}}} \]
that forgets from $C_{p^k}$-equivariant homotopy classes to the underlying $C_{p^{k - 1}}$-equivariant homotopy classes of maps.  Writing $Q=N^{C_{p^{k-1}n}}_{C_n}H\m{R}$ as above, this means that $i^*_{C_{p^{k-1}n}}Fd\lambda[a]_{k}$ is the $C_{p^{k - 1}}$ homotopy class of the map
\begin{multline}\label{FDlambdaxkcompositemap}
S^1_{pn\bullet}\sma N^{C_{p^{k}}}_e\sphere 
\cong S^1_{pn\bullet}\sma N^{C_{p^k}}_{C_{p^{k-1}}}\sphere_{C_{p^{k-1}}} 
\xto{N[a]_{k-1}}
 S^1_{pn\bullet}\sma N^{C_{p^{k}}}_{C_{p^{k-1}}}i^*_{C_{p^{k-1}}}Q
\xto{\varepsilon\circ \eta} \\
 S^1_{pn\bullet}\sma i_{C_{p^{k}}}^*N^{C_{p^{k}n}}_{C_{p^{k-1}n}} Q \xto{\eta_{S^1}}
  S^1_{pn\bullet}\sma i^*_{C_{p^{k}}} N^{\cy, C_{p^kn}}_\bullet (N^{C_{p^{k}n}}_{C_{p^{k-1}n}} Q)
 		\xto{\cong} \\
S^1_{pn\bullet}\sma \sd_p N^{\cy,C_{p^{k - 1}n}}_\bullet Q 
 		\xto{\alpha} \sd_{p}N^{\cy, C_{p^{k-1}n}}_\bullet Q 
 		\xto{d_0^{(\bullet+1)(p-1)}} i^*_{C_{p^{k }}} N^{\cy,C_{p^{k-1}n}}_\bullet Q
\end{multline}
where $d_0^{(\bullet+1)(p-1)}$ is the usual simplicial map that on realization is homotopic to the equivalence $\lvert \sd_{p}N^{\cy,C_{p^{k-1}n}}_\bullet Q\rvert\to \lvert N^{\cy, C_{p^{k-1}n}}_\bullet Q\rvert$. As in the previous case, on simplicial level $q$, the map $d_0^{(q+1)(p-1)}$ simply multiplies the left-most smash factors of $Q^{\sma (q+1)p}$ until there are only $q+1$-factors remaining. Unraveling definitions analogously to the case where $k = 1$ produces a diagram similar to (\ref{unpackingofFlambdaa1oncomponents}), with $a$ replaced by $[a]_{k-1}$.  Note that the map (\ref{FDlambdaxkcompositemap}) produces a product of $p-1$ images of $[a]_{k-1}$ under powers of the generator of $C_{p^{k-1}n}$ in the left-most smash factor of $Q$.

We now turn to the right hand side of the relation, $\lambda[a]_{k - 1}^{p-1}d\lambda[a]_{k - 1}$.  The element $\lambda[a]_{k - 1}^{p-1}d\lambda[a]_{k - 1}$ is represented by the following map in simplicial $C_{p^{k-1}}$-spectra, where $Q=N^{C_{p^{k-1}n}}_{C_n} H\m{R}$ as before and we have removed restrictions $i^*_{C_{p^{k-1}n}}$ from the notation for readability.
\begin{multline}\label{eq:lambdaakmult}
S^1_{n\bullet}\sma (\sphere_{C_{p^{k-1}}})^{\sma p} \xto{\id \sma [a]_{k-1}^{\sma p}}
 S^1_{n\bullet}\sma Q^{\sma p} 
	 \xto{\eta_{S^1}^{\sma p}} S^1_{n\bullet}\sma (N^{\cy, C_{p^{k-1}n}}_\bullet Q)^{\sma p} \\
	 \xto{\id\sma \text{mult} \sma \id} S^1_{n\bullet} \sma N^{\cy, C_{p^{k-1}n}}_\bullet Q \sma  N^{\cy, C_{p^{k-1}n}}_\bullet Q
	 \xto{\cong}  N_\bullet^{\cy, C_{p^{k-1}n}} Q \sma (S^1_{n\bullet} \sma N_\bullet^{\cy, C_{p^{k-1}n}} Q )	\\
 \xto{\id \sma \alpha}  N_\bullet^{\cy, C_{p^{k-1}n}} Q \sma  N_\bullet^{\cy, C_{p^{k-1}n}} Q 
	 \xto{\mu} N_\bullet^{\cy, C_{p^{k-1}n}} Q
\end{multline}
Here the $\text{mult}$ map multiplies the first $p-1$ copies of $Q=N_{C_n}^{C_{p^{k-1}n}} H \m{R}.$ 

Ultimately, $\lambda[a]_{k - 1}^{p-1}d\lambda[a]_{k - 1}$ is the homotopy class of the realization of the map that on simplicial level $q$ and on the wedge summand indexed by the generator of $C_{n(q+1)}$ sends 
\[\sphere^{\sma (q+1)}\xto{[a]_{k-1}^{p-1}\sma 1^{\sma (q-1)}\sma g^{-1}[a]_{k-1}}Q^{\sma q+1}\]
As before, while the maps (\ref{FDlambdaxkcompositemap}) and (\ref{eq:lambdaakmult}) do not agree on the nose, they do represent homotopy equivalent elements: once we geometrically realize and obtain the $S^1$-action, the action of an element of $C_{p^{k-1}n}\subset S^1$ is homotopic to the identity.

Finally, turning to the case of an element $a\in \m{R}(C_n/C_m)$, we note that such an element can be represented by a map
\[\sphere_{C_m}\to i^*_{C_m}H\m{R}.\]
The analysis of the relation goes through in the same way as in the case where $C_m=e$; we need only to work with $C_{pm}$-equivariance or $C_{p^{k}m}$-equivariance.  The key observation is that the isomorphism (\ref{Lambdapnopisomorphism}) is not only an isomorphism of simplicial  $C_{p}$-spectra but also of  $C_{pm}$-spectra because it arises from an isomorphism of $\Lambda_{pn}^\op$-spectra.
\end{proof}

\bibliographystyle{plain}
\bibliography{bib}

\begin{thebibliography}{10}

\bibitem{AGHKK2}
Katharine Adamyk, Teena Gerhardt, Kathryn Hess, Inbar Klang, and Hana~Jia Kong.
\newblock A shadow perspective on equivariant {H}ochschild homologies.
\newblock {\em Int. Math. Res. Not. IMRN}, (18):15299--15357, 2023.

\bibitem{AKGH}
Gabriel Angelini-Knoll, Teena Gerhardt, and Michael Hill.
\newblock Real topological {H}ochschild homology, norms, and {W}itt vectors.
\newblock {\em Preprint arXiv:2111.06970}, 2021.

\bibitem{AnBlGeHiLaMa}
Vigleik Angeltveit, Andrew~J. Blumberg, Teena Gerhardt, Michael~A. Hill, Tyler
  Lawson, and Michael~A. Mandell.
\newblock Topological cyclic homology via the norm.
\newblock {\em Doc. Math.}, 23:2101--2163, 2018.

\bibitem{AnGe11}
Vigleik Angeltveit and Teena Gerhardt.
\newblock On the algebraic {$K$}-theory of the coordinate axes over the
  integers.
\newblock {\em Homology Homotopy Appl.}, 13(2):103--111, 2011.

\bibitem{AG_ROS1TR}
Vigleik Angeltveit and Teena Gerhardt.
\newblock {$RO(S^1)$}-graded {TR}-groups of {$\Bbb F_p$}, {$\Bbb Z$} and
  {$\ell$}.
\newblock {\em J. Pure Appl. Algebra}, 215(6):1405--1419, 2011.

\bibitem{AGH_Ktrunc}
Vigleik Angeltveit, Teena Gerhardt, and Lars Hesselholt.
\newblock On the {$K$}-theory of truncated polynomial algebras over the
  integers.
\newblock {\em J. Topol.}, 2(2):277--294, 2009.

\bibitem{BMS}
Bhargav Bhatt, Matthew Morrow, and Peter Scholze.
\newblock Topological {H}ochschild homology and integral {$p$}-adic {H}odge
  theory.
\newblock {\em Publ. Math. Inst. Hautes \'{E}tudes Sci.}, 129:199--310, 2019.

\bibitem{BlGeHiLa}
Andrew~J. Blumberg, Teena Gerhardt, Michael~A. Hill, and Tyler Lawson.
\newblock The {W}itt vectors for {G}reen functors.
\newblock {\em J. Algebra}, 537:197--244, 2019.

\bibitem{MR3773736}
Andrew~J. Blumberg and Michael~A. Hill.
\newblock Incomplete {T}ambara functors.
\newblock {\em Algebr. Geom. Topol.}, 18(2):723--766, 2018.

\bibitem{BlumbergMandell-cycl}
Andrew~J. Blumberg and Michael~A. Mandell.
\newblock The homotopy theory of cyclotomic spectra.
\newblock {\em Geom. Topol.}, 19(6):3105--3147, 2015.

\bibitem{BohmannRiehl}
Anna~Marie Bohmann.
\newblock A comparison of norm maps.
\newblock {\em Proc. Amer. Math. Soc.}, 142(4):1413--1423, 2014.
\newblock With an appendix by Bohmann and Emily Riehl.

\bibitem{bokstedt_thh}
M.~B{\"o}kstedt.
\newblock Topological {H}ochschild homology.
\newblock 1985.

\bibitem{BHM}
M.~B{\"o}kstedt, W.~C. Hsiang, and I.~Madsen.
\newblock The cyclotomic trace and algebraic ${K}$-theory of spaces.
\newblock {\em Invent. Math.}, 111(3):465--539, 1993.

\bibitem{BokstedtMadsen}
M.~B\"{o}kstedt and I.~Madsen.
\newblock Topological cyclic homology of the integers.
\newblock Number 226, pages 7--8, 57--143. 1994.
\newblock $K$-theory (Strasbourg, 1992).

\bibitem{BokstedtMadsen2}
M.~B\"{o}kstedt and I.~Madsen.
\newblock Algebraic {$K$}-theory of local number fields: the unramified case.
\newblock In {\em Prospects in topology ({P}rinceton, {NJ}, 1994)}, volume 138
  of {\em Ann. of Math. Stud.}, pages 28--57. Princeton Univ. Press, Princeton,
  NJ, 1995.

\bibitem{Bouc_GreenGSets}
Serge Bouc.
\newblock {\em Green functors and {$G$}-sets}, volume 1671 of {\em Lecture
  Notes in Mathematics}.
\newblock Springer-Verlag, Berlin, 1997.

\bibitem{Bouc}
Serge Bouc.
\newblock Non-additive exact functors and tensor induction for {M}ackey
  functors.
\newblock {\em Mem. Amer. Math. Soc.}, 144(683):viii+74, 2000.

\bibitem{BrunWittTambara}
Morten Brun.
\newblock Witt vectors and {T}ambara functors.
\newblock {\em Adv. Math.}, 193(2):233--256, 2005.

\bibitem{BDS18}
Morten Brun, Bj{\o}rn Dundas, and Martin Stolz.
\newblock {Equivariant Structure on Smash Powers}.
\newblock {\em arXiv e-prints}, 2018.
\newblock arxiv:1604.059939.

\bibitem{ChanBiIncomplete}
David Chan.
\newblock Bi-incomplete {T}ambara functors as {$\mathcal{O}$}-commutative
  monoids.
\newblock {\em Tunis. J. Math.}, 6(1):1--47, 2024.

\bibitem{CGK}
David Chan, Teena Gerhardt, and Inbar Klang.
\newblock Trace methods for equivariant algebraic {$K$}-theory.
\newblock {\em Preprint}, 2024.

\bibitem{Costeanu}
Viorel Costeanu.
\newblock On the 2-typical de {R}ham-{W}itt complex.
\newblock {\em Doc. Math.}, 13:413--452, 2008.

\bibitem{Dress73}
Andreas W.~M. Dress.
\newblock Contributions to the theory of induced representations.
\newblock In {\em Algebraic {$K$}-theory, {II}: ``{C}lassical'' algebraic
  {$K$}-theory and connections with arithmetic ({P}roc. {C}onf., {B}attelle
  {M}emorial {I}nst., {S}eattle, {W}ash., 1972)}, volume Vol. 342 of {\em
  Lecture Notes in Math.}, pages 183--240. Springer, Berlin-New York, 1973.

\bibitem{Hesselholt1996}
Lars Hesselholt.
\newblock On the {$p$}-typical curves in {Q}uillen's {$K$}-theory.
\newblock {\em Acta Math.}, 177(1):1--53, 1996.

\bibitem{HesselholtLecWitt}
Lars Hesselholt.
\newblock Lecture notes on {W}itt vectors, 2008.

\bibitem{HeMa97}
Lars Hesselholt and Ib~Madsen.
\newblock On the {$K$}-theory of finite algebras over {W}itt vectors of perfect
  fields.
\newblock {\em Topology}, 36(1):29--101, 1997.

\bibitem{HeMa03}
Lars Hesselholt and Ib~Madsen.
\newblock On the {$K$}-theory of local fields.
\newblock {\em Ann. of Math. (2)}, 158(1):1--113, 2003.

\bibitem{HeMa04}
Lars Hesselholt and Ib~Madsen.
\newblock On the {D}e {R}ham-{W}itt complex in mixed characteristic.
\newblock {\em Ann. Sci. \'{E}cole Norm. Sup. (4)}, 37(1):1--43, 2004.

\bibitem{HillHopkins}
M.~A. Hill and M.~J. Hopkins.
\newblock Equivariant symmetric monoidal structures.
\newblock arxiv.org: 1610.03114, 2016.

\bibitem{HHR}
M.~A. Hill, M.~J. Hopkins, and D.~C. Ravenel.
\newblock On the nonexistence of elements of {K}ervaire invariant one.
\newblock {\em Ann. of Math. (2)}, 184(1):1--262, 2016.

\bibitem{HiMeQu}
Michael~A. Hill, David Mehrle, and James~D. Quigley.
\newblock Free incomplete {T}ambara functors are almost never flat.
\newblock {\em Int. Math. Res. Not. IMRN}, (5):4225--4291, 2023.

\bibitem{Hoy}
Rolf Hoyer.
\newblock Two topics in stable homotopy theory.
\newblock {\em PhD thesis, University of Chicago}, 2014.

\bibitem{KMN}
Achim Krause, Jonas McCandless, and Thomas Nikolaus.
\newblock Polygonic spectra and {TR} with coefficients.
\newblock {\em Preprint arXiv:2302.07686}, 2023.

\bibitem{MThesis}
Kristen Mazur.
\newblock {\em On the structure of Mackey functors and Tambara functors}.
\newblock PhD thesis, University of Virginia, 2013.

\bibitem{MSV}
J.~McClure, R.~Schw{\"a}nzl, and R.~Vogt.
\newblock ${T}{H}{H}({R})\cong {R}\otimes {S}\sp 1$ for ${E}\sb \infty$ ring
  spectra.
\newblock {\em J. Pure Appl. Algebra}, 121(2):137--159, 1997.

\bibitem{Merling}
Mona Merling.
\newblock Equivariant algebraic {K}-theory of {$G$}-rings.
\newblock {\em Math. Z.}, 285(3-4):1205--1248, 2017.

\bibitem{strickland2012tambara}
Neil Strickland.
\newblock Tambara functors, 2012.

\bibitem{Tambara_multitransfer}
D.~Tambara.
\newblock On multiplicative transfer.
\newblock {\em Comm. Algebra}, 21(4):1393--1420, 1993.

\bibitem{ullman2013tambara}
John Ullman.
\newblock Tambara functors and commutative ring spectra, 2013.

\end{thebibliography}

\end{document}